\documentclass[12pt,a4paper]{amsart}
\usepackage{amssymb, amsmath, amsthm, amsfonts,xcolor, enumerate,bbm, hyperref, fullpage}
\usepackage{tikz}
\usetikzlibrary{patterns}
\usetikzlibrary{shapes}
\newtheorem{theorem}{Theorem}[section]
\newtheorem{conj}[theorem]{Conjecture}
\newtheorem{lemma}[theorem]{Lemma}
\newtheorem{cor}[theorem]{Corollary}
\newtheorem{prop}[theorem]{Proposition}
\newtheorem{claim}[theorem]{Claim}
\theoremstyle{definition}

\usepackage{lipsum}
\usepackage{paralist}

\theoremstyle{definition}

\theoremstyle{definition}

\theoremstyle{definition}

\theoremstyle{definition}
\newtheorem{defn}[theorem]{Definition}
\theoremstyle{definition}

\theoremstyle{definition}
\newtheorem{obs}[theorem]{Observation}
\theoremstyle{definition}

\newcommand{\bnk}{\binom{n}{k}}
\newcommand{\bekr}{\binom{n-1}{k-1}}
\newcommand{\bnkk}{\binom{n-k-1}{k-1}}

\newcommand{\card}[1]{\left| #1 \right|}
\newcommand{\floor}[1]{\lfloor #1 \rfloor}
\newcommand{\ep}{\varepsilon}

\newcommand{\lam}{\lambda}
\newcommand{\al}{\alpha}
\newcommand{\be}{\beta}
\newcommand{\de}{\delta}
\newcommand{\si}{\sigma}
\newcommand{\ga}{\gamma}

\newcommand{\Ga}{\Gamma}
\newcommand{\De}{\Delta}
\newcommand{\bR}{\mathbb{R}}
\newcommand{\Ho}{\H{o}}

\newcommand{\1}{\mathbbm{1}}

\newcommand{\bN}{\ensuremath{\mathbb{N}}}

\renewcommand{\l}{\ensuremath{\ell}}

\newcommand{\ept}{\mathbb{E}}
\DeclareMathOperator{\disj}{dp}

\newcommand{\mc}[1]{\mathcal{#1}}

\title{Structure and supersaturation for intersecting families}
\author{J\'ozsef Balogh, Shagnik Das, Hong Liu, Maryam Sharifzadeh and Tuan Tran}
\thanks{J.B.\ was partially supported by NSF Grant DMS-1500121 and by the Langan Scholar Fund (UIUC).  S.D.\ was supported by GIF grant G-1347-304.6/2016.  H.L.\ was supported by the Leverhulme Trust Early Career Fellowship~ECF-2016-523. M.S.\ was supported by ERC grant~306493 and the European Union’s Horizon 2020 research and innovation programme under the Marie Curie grant agreement No 752426. T.T.\ 
was supported by the  Alexander Humboldt Foundation, and by the GACR grant GJ16-07822Y, with institutional support RVO:67985807.}

\begin{document}
\maketitle

\begin{abstract}
The extremal problems regarding the maximum possible size of intersecting families of various combinatorial objects have been extensively studied. In this paper, we investigate supersaturation extensions, which in this context ask for the minimum number of disjoint pairs that must appear in families larger than the extremal threshold. We study the minimum number of disjoint pairs in families of permutations and in $k$-uniform set families, and determine the structure of the optimal families. Our main tool is a removal lemma for disjoint pairs. We also determine the typical structure of $k$-uniform set families without matchings of size $s$ when $n \ge 2sk + 38s^4$, and show that almost all $k$-uniform intersecting families on vertex set $[n]$ are trivial when $n\ge (2+o(1))k$.
\end{abstract}

\section{Introduction}
Determining the size of intersecting families of discrete objects is a line of research with a long history, originating in extremal set theory. A set family is \emph{intersecting} if any two of its sets share a common element. A classic result of Erd\Ho s, Ko and Rado~\cite{ekr61} from 1961 states that when $n\ge 2k$, the size of the largest intersecting $k$-uniform set family 
over $[n]$ is ${n-1\choose k-1}$. Furthermore, when $n\ge 2k+1$, the only extremal configurations are the \emph{trivial} families, where all edges contain a given element. This fundamental theorem has since inspired a great number of extensions and variations.

A recent trend in extremal combinatorics is to study the supersaturation extension of classic results. This problem, sometimes referred to as the Erd\H{o}s--Rademacher problem, asks for the number of forbidden substructures that must appear in a configuration larger than the extremal threshold. We often observe an interesting phenomenon: while the extremal result only requires one forbidden substructure to appear, we usually find several. The first such line of research extended Mantel's Theorem~\cite{Mantel}, which states that an $n$-vertex triangle-free graph can have at most $\lfloor n^2/4\rfloor$ edges. Rademacher (unpublished) showed that one additional edge would force the appearance of at least $\lfloor n/2\rfloor$ triangles. Determining the number of triangles in larger graphs attracted a great deal of attention, starting with the works of Erd\H{o}s~\cite{E-R, E-clique} and Lov\'asz and Simonovits~\cite{L-S} and culminating in the asymptotic solution due to Razborov~\cite{Raz} and the recent exact solution determined by Liu, Pikhurko and Staden~\cite{LPS17}.  Supersaturation problems have since been studied in various contexts; examples include extremal graph theory~\cite{B-L, K-N, Mu, Niki, PY, Rei}, extremal set theory~\cite{BW, D-G-S, DGKS, Kl, Sam}, poset theory~\cite{BPW, NSS, SS1}, and group theory \cite{C-P-S,H14,SS2}.

The first result of our paper concerns supersaturation for the extension of the Erd\H{o}s--Ko--Rado Theorem to families of permutations.  A pair of permutations $\sigma,\pi\in S_n$ is said to be {\em intersecting} if $\{i\in [n]:\pi(i)=\si(i)\} \neq \emptyset$, and {\em disjoint} otherwise. A family $\mc{F}\subseteq S_n$ is {\em intersecting} if every pair of permutations in the  family is. A natural construction of an intersecting family is to fix some pair $i,j\in [n]$, and take all permutations that map $i$ to $j$; we call this a {\em coset}, and denote it by $\mc T_{(i,j)}$.  Observe that $\card{\mc T_{(i,j)}} = (n-1)!$, and Deza and Frankl~\cite{DF} showed that this is the largest possible size of an intersecting family in $S_n$.

In the corresponding supersaturation problem, we seek to determine how many disjoint pairs of permutations must appear in larger families. We write $\disj(\mc{F})$ for the number of disjoint pairs of permutations in a family $\mc{F}\subseteq S_n$. By the Deza--Frankl Theorem, when $\card{\mc{F}}\le (n-1)!$, we need not have any disjoint pairs in $\mc{F}$, while for $\card{\mc{F}}>(n-1)!$, $\disj(\mc{F})$ must be positive.

One might expect a family of permutations with the minimum number of disjoint pairs to contain large intersecting subfamilies, and a candidate construction is therefore the union of an appropriate number of cosets.  However, these unions are not isomorphic, as pairs of cosets can intersect each other differently.  Indeed, given pairs $(i_1, j_1) \neq (i_2, j_2) \in [n]^2$, we have $\mc T_{(i_1, j_1)} \cap \mc T_{(i_2, j_2)} = \{ \pi \in S_n : \pi(i_1) = j_1 \textrm{ and } \pi(i_2) = j_2 \}$, which is empty if $i_1 = i_2$ or $j_1 = j_2$, and has size $(n-2)!$ otherwise.  To fit the family within as few cosets as possible, we should take the cosets to be pairwise-disjoint, motivating the following definition.

\begin{defn}[The family $\mc T(n,s)$] \label{def:lexperm}
	Writing $\pi = \begin{pmatrix} \pi(1) & \pi(2) & \hdots & \pi(n) \end{pmatrix}$, we can equip $S_n$ with the lexicographic ordering, where $\pi < \sigma$ if there is some $k \in [n]$ such that $\pi(k) < \sigma(k)$ and $\pi(i) = \sigma(i)$ for all $i < k$.  Then, given any $0 \le s \le n!$, we denote by $\mc T(n,s)$ the first $s$ permutations under this ordering.
	
	In particular, if $s = \ell (n-1)! + r$ for some $0 \le \ell \le n$ and $0 \le r < (n-1)!$, the family $\mc T(n,s)$ contains the pairwise-disjoint cosets $\mc T_{(1,j)}$ for $1 \le j \le \ell$, together with $r$ further permutations from the disjoint coset $\mc T_{(1,\ell+1)}$.
\end{defn}

 Our first result shows, for certain ranges of family sizes $s$, that these families indeed minimise $\disj(\mc{F})$ over all families $\mc F \subseteq S_n$ with $\card{\mc F} = s$.

\begin{theorem}\label{thm-supsat-perm}
	There exists a constant $c>0$ such that the following holds. Let $n,k$ and $s$ be positive integers such that $k\le cn^{1/2}$, and $s=(k+\ep)(n-1)!$ for some real $\ep$ with $|\ep|\le ck^{-3}$. Then any family $\mc{F}\subseteq S_n$ with $\card{\mc{F}}=s$ satisfies $\disj(\mc{F})\ge \disj(\mc{T}(n,s))$.
\end{theorem}

We next consider the supersaturation extension of the original Erd\Ho s--Ko--Rado Theorem, where one seeks to minimise the number of disjoint pairs of sets in a $k$-uniform family of $s$ subsets of $[n]$.  Bollob\'as and Leader~\cite{Bol-L} provided, for every $s$, a family of constructions known as the $\ell$-balls, and conjectured that for some $1 \le \ell \le k$, an $\ell$-ball is optimal for the supersaturation problem.  In particular, when $\ell = 1$, the construction is an initial segment of the lexicographic ordering.

Denote by $\binom{[n]}{k}$ the family of all $k$-element subsets of $[n]$. Letting $\mc L(n,k,s)$ be the initial segment of the first $s$ sets in $\binom{[n]}{k}$, we write $\disj(n,k,s)$ for $\disj(\mc L(n,k,s))$, where again $\disj(\mc F)$ is the number of disjoint pairs in a set family $\mc F$.  Das, Gan and Sudakov~\cite{D-G-S-2} proved that if $n > 108(k^3r + k^2r^2)$ and $s \le \binom{n}{k} - \binom{n-r}{k}$, then for any family $\mc F \subseteq \binom{[n]}{k}$ of size $s$, $\disj(\mc F) \ge \disj(n,k,s)$.  That is, when $n$ is sufficiently large and the families are of small size, the initial segments of the lexicographic order minimise the number of disjoint pairs, confirming the Bollob\'as--Leader conjecture in this range.

Note that, for fixed $r$, the result in~\cite{D-G-S-2} requires $k = O(n^{1/3})$.  Frankl, Kohayakawa and R\"odl~\cite{FKR} showed that initial segments of the lexicographic order are \emph{asymptotically} optimal even for larger uniformities $k$.  In our next result, we extend the exact results to larger $k$ as well, showing that the lexicographic initial segments are still optimal when $k = O(n^{1/2})$.

\begin{theorem} \label{thm:setsupersat}
There is some absolute constant $C$ such that if $n \ge C k^2 r^3$ and $s \le \bnk - \binom{n-r}{k}$, then any family $\mc{F} \subseteq \binom{[n]}{k}$ with $\card{\mc{F}} = s$ satisfies $\disj(\mc{F}) \ge \disj(n,k,s)$; that is, $\mc{L}(n,k,s)$ minimises the number of disjoint pairs.
\end{theorem}

With our next results, we address a different variation of classic extremal problems.  Rather than considering the supersaturation phenomenon, we describe the typical structure of set families with a given property, showing that almost all such families are subfamilies of the trivial extremal constructions.

We first consider the famous Erd\H{o}s Matching Conjecture concerning the largest $k$-uniform set families over a ground set of size $n$ that have no matching of size $s$. There are two constructions that trivially avoid a matching of size $s$: a clique on $ks-1$ vertices, and the family of all edges intersecting a set of size $s-1$.  In~\cite{erd65}, Erd\H{o}s conjectured that one of these constructions is always optimal.

\begin{conj}[Erd\H{o}s~\cite{erd65}, 1965]
Given integers $n, k$ and $s$, let $\mc F \subseteq \binom{[n]}{k}$ be a set family with no matching of size $s$.  Then
\[ \card{\mc F} \le \max \left( \binom{ks-1}{k}, \binom{n}{k} - \binom{n-s+1}{k} \right). \]
\end{conj}

Frankl~\cite{Frankl13} proved the conjecture in the range $n \ge (2s-1)k - s + 1$, showing that the extremal families can be covered by $s-1$ elements.  Adapting the methods of Balogh et al.~\cite{BDDLS}, we show that a slightly larger lower bound on $n$ guarantees that almost all families without a matching of size $s$ have a cover of size $s-1$.

\begin{theorem}\label{thm:no-s-disj}
Let $n,k \ge 3$ and $s\ge 2$ be integers with $n\ge 2sk+38s^4$. Then the number of subfamilies of $\binom{[n]}{k}$ with no matching of size $s$ is
	$\left(\binom{n}{s-1}+o(1)\right)2^{\bnk-\binom{n-s+1}{k}}$, where the term $o(1)$ tends to $0$ as $n \rightarrow \infty$.
\end{theorem}

The $s = 2$ case corresponds to intersecting families.  In this case, Balogh et al.~\cite{BDDLS} showed that when $n \ge (3 + o(1))k$, almost all intersecting families are trivial.  Our final result improves the required bound on $n$ to the asymptotically optimal $n \ge (2 + o(1))k$.  Indeed, when $n=2k$, then the number of intersecting families is $3^{\frac{1}{2}{n\choose k}}=3^{{n-1\choose k-1}}$, since we can freely choose at most one set from each complementary pair of $k$-sets $\left\{ A, [n] \setminus A \right\}$.
\begin{theorem}\label{thm:aa}
There exists a positive constant $C$ such that for $k\ge 2$ and $n \ge 2k + C \sqrt{k \ln k}$, almost all intersecting families in ${[n]\choose k}$ are trivial. In particular, the number of intersecting families in $\binom{[n]}{k}$ is $(n+o(1))2^{\bekr}$, where the term $o(1)$ tends to $0$ as $n\rightarrow \infty$.
\end{theorem}

\noindent\textbf{Remark:}
	During the preparation of this paper, Theorem~\ref{thm:aa} (with a superior constant $C=2$) was proven independently by Frankl and Kupavskii~\cite{fk17-counting} using different methods.

\medskip

\noindent\textbf{Outline and notation.} The rest of the paper is organised as follows. We discuss families of permutations in Section~\ref{sec-perm}, in particular proving the supersaturation result of  Theorem~\ref{thm-supsat-perm}.  Section~\ref{sec:uniform} is devoted to supersaturation for set families and the proof of Theorem~\ref{thm:setsupersat}. In Section~\ref{sec:aa}, we address the typical structure of families, proving Theorems~\ref{thm:no-s-disj} and~\ref{thm:aa}. Section~\ref{sec:concluding} contains some concluding remarks, including a counterexample to the Bollob\'as--Leader conjecture.

We use standard set-theoretic and asymptotic notation. We write $\binom{X}{k}$ for the family of all $k$-element subsets of a set $X$. Given two functions $f$ and $g$ of some underlining parameter $n$, if $\lim_{n\rightarrow \infty}f(n)/g(n)=0$, we write $f=o(g)$. For $a,b,c\in\bR_{+}$, we write $a=b\pm c$ if $b-c\le a\le b+c$.

\section{Supersaturation for families of permutations}\label{sec-perm}
In this section, we study the supersaturation problem concerning the number of disjoint pairs in a family of permutations. Our main tool is a removal lemma for disjoint pairs of permutations, showing that families with relatively few disjoint pairs are close to unions of cosets. We start by collecting some basic facts.

\subsection{The derangement graph} \label{subsec:derangement}
Let $S_n$ be the symmetric group on $[n]$. A permutation $\tau \in S_n$ is called a \emph{derangement} if $\tau(i)\neq i$ for every $i\in[n]$. Let $\mc{D}_n$ be the set of all derangements in $S_n$. Denote by $\Ga_n$ the \emph{derangement graph} on $S_n$, that is, $\sigma \sim \pi$ if $\sigma\cdot \tau=\pi$ for some $\tau\in \mc{D}_n$. In other words, $\sigma$ and $\tau$ are adjacent in $\Ga_n$ if and only if they are disjoint.

We denote by $d_n$ the number of derangements in $S_n$. By construction, $\Ga_n$ is a $d_n$-regular graph. A standard application of the inclusion--exclusion principle shows
\[ d_n = \card{ \mc D_n } = n! \sum_{i=0}^{n} \frac{(-1)^i}{i!} \sim \frac{n!}{e}.\]
We also introduce the notation $D_n = d_n + d_{n-1}$, which we will use to keep track of disjoint pairs in certain subgraphs of the derangement graph.

For instance, consider the subgraph of $\Ga_n$ induced by two disjoint cosets $\mc{T}_{(i_1,j)}$ and $\mc{T}_{(i_2,j)}$. Since the cosets are intersecting families, they are independent sets in $\Ga_n$, and so $\Ga_n[\mc T_{(i_1,j)}, \mc T_{(i_2,j)}]$ is bipartite. For any $\sigma \in \mc{T}_{(i_1,j)}$ and any neighbour $\pi = \sigma \cdot \tau$, where $\tau\in\mc{D}_n$, we have $\pi \in \mc{T}_{(i_2, j)}$ if and only if $\tau(i_2)=i_1$. It is straightforward to see that there are $d_{n-2}$ such derangements $\tau$ with $\tau(i_1)=i_2$ and $d_{n-1}$ derangements with $\tau(i_1)\neq i_2$. As a result, every vertex of the bipartite graph has the same degree $d_{n-2}+d_{n-1}=D_{n-1}$.

Repeating this argument for all the cosets of the form $\mc T_{(i',j)}$ gives the recurrence relation
\begin{equation} \label{eqn-dn-recurrence}
d_n = (n-1)D_{n-1} = (n-1)(d_{n-1} + d_{n-2}).
\end{equation}

For our investigation we shall need some information on the spectrum of the derangement graph $\Ga_n$. Since $\Ga_n$ is $d_n$-regular, its largest eigenvalue $\lam_0$ is $d_n$ with constant eigenvector $\vec{\textbf{1}}=\begin{pmatrix} 1 & \hdots & 1 \end{pmatrix}$.  We shall order the eigenvalues $\{\lambda_0, \lambda_1, \hdots, \lambda_{n!-1} \}$ in a (perhaps non-standard) way, so that $d_n = \card{\lambda_0} \ge \card{\lambda_1} \ge \card{\lambda_2} \ge \hdots \ge \card{\lambda_{n!-1}}$.  Rentner~\cite{R07} showed $\lambda_1 = -d_n / (n-1)$, while Ellis~\cite{E12} proved there is some positive constant $K$ such that
\begin{eqnarray}\label{eq-const-K}
	\card{\lambda_2} \le Kd_n/n^2.
\end{eqnarray}
Furthermore, as shown by Ellis, Friedgut and Pilpel~\cite{EFP}, the span of the $\lam_0$- and $\lam_1$-eigenspaces is $U_1=\hbox{span}\{\1_{\mc{T}_{(i,j)}}:i,j \in [n]\}$, the span of the characteristic vectors of the cosets.
\subsection{A removal lemma}

For any integer $s > \frac12 (n-1)!$, there are unique $k \in \mathbb{N}$ and $\ep \in \left( -\frac12, \frac12 \right]$ such that $s = (k+ \ep) (n-1)!$. For this choice of $s$, the family $\mc{T}(n,s)$ from Definition~\ref{def:lexperm} is a subfamily of $S_n$ consisting of $\floor{k+\ep}$ pairwise disjoint cosets and $(k+\ep-\floor{k+\ep})(n-1)!$ permutations from another disjoint coset.  Hence,
\begin{align}\label{est-T}
\disj(\mc{T}(n,s)) &= \left( \binom{\floor{k+\ep}}{2} + \floor{k+\ep}(k + \ep - \floor{k + \ep}) \right) (n-1)! D_{n-1} \notag \\
	&= \left({k\choose 2}+\left(k-\frac{1}{2}\right)\ep+\frac{1}{2}|\ep|\right)(n-1)!D_{n-1},
\end{align}
as $\disj(\mc{F})=e(\Gamma_n[\mc{F}])$, and the bipartite subgraphs of $\Ga_n$ induced by disjoint cosets are $D_{n-1}$-regular.

We will now prove a removal lemma for disjoint pairs of permutations, which states that any family $\mc{F}\subseteq S_n$ of size $s \approx k(n-1)!$ with $\disj(\mc{F})\approx \disj(\mc{T}(n,s))$ must be `close' to a union of $k$ cosets.

\begin{lemma} \label{lem-removal-perm}
There exist positive constants $C$ and $c$ such that the following holds for sufficiently large $n$. Let $1 \le k < n/2$ be an integer, and let $\ep \in \bR$ and $\be\in\bR_{+}$ be such that $\max\{|\ep|,\be\} \le ck$. If $\mc{F} \subseteq S_n$ is a family of size $s = (k + \ep)(n-1)!$ and $\disj(\mc{F}) \le \disj(\mc{T}(n,s))+\be(n-1)!D_{n-1}$, then there is some union $\mc{G}$ of $k$ cosets with the property that
\[
\card{\mc{F}\De\mc{G}}\le Ck^2\left(\frac{1}{n}+\sqrt{\frac{6(|\ep|+\be)}{k}}\right)(n-1)!.
\]
\end{lemma}

In the proof of Lemma \ref{lem-removal-perm} we shall use a stability result due to Ellis, Filmus and Friedgut~\cite[Theorem 1]{EFF}. To state their theorem we need some additional notation. We equip $S_n$ with the uniform distribution.  Then, for any function $f:S_n\rightarrow \bR$, the expected value of $f$ is defined by $\ept[f]=\frac{1}{n!}\sum_{\si\in S_n}f(\si)$. The inner product of two functions $f,g:S_n\rightarrow \bR$ is defined as $\langle f,g \rangle = \ept[fg] = \frac{1}{n!} \sum_{\si \in S_n} f(\si) g(\si)$; this induces the norm $\|f\|=\sqrt{\langle f,f \rangle}$. Given $c>0$, let $\mathrm{round}(c)$ denote the nearest integer to $c$.

\begin{theorem}[Ellis, Filmus and Friedgut]\label{thm-stab}
There exist positive constants $C_0$ and $\de_0$ such that the following holds. Let $\mc{F}$ be a subfamily of $S_n$ with $\card{\mc{F}}=\al(n-1)!$ for some $\al\le n/2$. Let $f=\1_{\mc{F}}$ be the characteristic function of $\mc{F}$ and let $f_{U_1}$ be the orthogonal projection of $f$ onto $U_1$. If $\ept[(f-f_{U_1})^2]=\de\ept[f]$ for some $\de\le\de_0$, then 
\[
\ept[(f-g)^2]\le C_0\al^2(1/n+\de^{1/2})/n,
\]
where $g$ is the characteristic function of a union of $\mathrm{round}(\al)$ cosets of $S_n$.
\end{theorem}

We now derive the removal lemma from Theorem \ref{thm-stab}.

\begin{proof}[Proof of Lemma \ref{lem-removal-perm}]	
Set $c=\min\{\frac{\de_0}{12},\frac12\}$ and $C=3C_0$, where $\de_0$ and $C_0$ are the positive constants from Theorem~\ref{thm-stab}. Let $f$ be the characteristic vector of $\mc{F}$. Write $f=f_0+f_1+f_2$, where $f_i$ is the projection of $f$ onto the $\lam_i$-eigenspace for $i=0,1$. By the orthogonality of the eigenspaces, 
\begin{eqnarray}\label{eq-orth}
	\|f\|^2=\|f_0\|^2+\|f_1\|^2+\|f_2\|^2.
\end{eqnarray}
Since $f$ is Boolean, 
\begin{equation}
\label{eq-ff0}
\|f\|^2=\ept[f^2]=\ept[f]=\frac{\card{\mc{F}}}{n!}=\frac{k+\ep}{n}\quad \mbox{and }\quad \|f_0\|^2=\langle f,\vec{\textbf{1}}\rangle^2 =\ept[f]^2=\left(\frac{k+\ep}{n}\right)^2.
\end{equation}
Let $A$ be the adjacency matrix of the derangement graph $\Ga_n$. Then
\[
2\disj(\mc{F})=2e(\Ga_n[\mc{F}])=f^TAf=\sum_{i=0}^2f_i^TAf_i\stackrel{(\ref{eq-const-K})}{\ge}\lam_0f_0^Tf_0+\lam_1f_1^Tf_1-\frac{Kd_n}{n^2}f_2^Tf_2.
\]
Dividing both sides by $n!$, we obtain the following inequalities when $n \ge 4K$:
\begin{eqnarray}\label{eq-dp-l}
     \frac{2 \disj(\mc{F})}{n!} &\ge& \lam_0 \| f_0 \|^2 + \lam_1 \| f_1 \|^2 - \frac{Kd_n}{n^2} \| f_2 \|^2 \nonumber\\
     &\stackrel{\eqref{eq-orth}}{=}& \lam_0 \| f_0 \|^2 + \lam_1 (\| f \|^2-\| f_0 \|^2-\| f_2 \|^2) - \frac{Kd_n}{n^2} \| f_2 \|^2\nonumber\\
     &=& (\lam_0 - \lam_1) \| f_0 \|^2 + \lam_1 \|f\|^2+\left( |\lam_1| - \frac{Kd_n}{n^2} \right) \| f_2 \|^2 \nonumber\\
      &\stackrel{\eqref{eq-ff0}}{\ge}& \frac{nd_n}{n-1}\cdot\left(\frac{k+\ep}{n}\right)^2-\frac{d_n}{n-1}\cdot\left(\frac{k+\ep}{n}\right)+\frac{3d_n}{4(n-1)}\cdot\| f_2 \|^2\nonumber\\
      &=& \frac{d_n}{n(n-1)}\cdot(k+\ep)(k+\ep-1)+\frac{3d_n}{4(n-1)}\cdot\| f_2 \|^2\nonumber\\
      &\stackrel{\eqref{eqn-dn-recurrence}}{=}&\frac{2D_{n-1}}{n}\cdot\left({k\choose 2}+\left(k-\frac{1}{2}\right)\ep+\frac{\ep^2}{2}\right)+\frac{3}{4}D_{n-1}\cdot\| f_2 \|^2.
\end{eqnarray}
On the other hand, by assumption we have
\[
\disj(\mc{F})\le \disj(\mc{T}(n,s))+\beta(n-1)!D_{n-1}\stackrel{\eqref{est-T}}{=}\left({k\choose 2}+\left(k-\frac{1}{2}\right)\ep+\frac{1}{2}|\ep|+\beta\right)(n-1)!D_{n-1}.
\]
Combined with~\eqref{eq-dp-l}, we get
\begin{eqnarray}\label{eq-f2}
	 \| f_2 \|^2 \le \frac{4}{3}\cdot \frac{|\ep|+2\be-\ep^2}{n}\le  \frac{3(|\ep|+\be)}{n}.
\end{eqnarray}
Moreover, $\ept[f]=\frac{k+\ep}{n} \ge \frac{k}{2n}$, as $|\ep| \le ck \le \frac{k}{2}$. Therefore, 
\[
\ept[(f-f_{U_1})^2]=\ept[(f-f_0-f_1)^2]=\|f_2\|^2 \le \frac{6(|\ep|+\be)}{k}\cdot\ept[f].
\]
Since $\frac{6(|\ep|+\be)}{k} \le 12c \le \de_0$, we may apply Theorem~\ref{thm-stab} to conclude that there exists a union $\mc{G}$ of $k$ cosets in $S_n$ such that
\[
	\ept[(f-\1_{\mc{G}})^2] \le \frac{C_0 (k+\ep)^2}{n} \left( \frac{1}{n} +\sqrt{\frac{6(|\ep|+\be)}{k}} \right)\le \frac{C k^2}{n}\left(\frac{1}{n}+\sqrt{\frac{6(|\ep|+\be)}{k}}\right).
\]
This gives $\card{\mc{F} \De \mc{G}}=\ept[(f-\1_{\mc{G}})^2]\cdot n! \le Ck^2\left(\frac{1}{n}+\sqrt{\frac{6(|\ep|+\be)}{k}}\right)(n-1)!$, completing our proof.
\end{proof}

We will use this removal lemma to prove Theorem~\ref{thm-supsat-perm} (a supersaturation result for disjoint pairs in $S_n$) in Subsection~\ref{sec:supsat-perm}.
However, from the proof above we can immediately deduce that for any\footnote{In the proof of Lemma~\ref{lem-removal-perm}, we used that $n$ was sufficiently large to bound $\frac{K d_n}{n^2} \| f_2 \|^2$.  However, in Proposition~\ref{claim:largeexact}, we have $\| f_2 \|^2 = 0$, and so do not require $n$ to be large.} $1 \le k \le n$, the union of $k$ pairwise disjoint cosets minimises the number of disjoint pairs among all families of $k(n-1)!$ permutations.

\begin{prop} \label{claim:largeexact}
For any positive integers $1 \le k \le n$, the family $\mc T = \cup_{j=1}^{k}\mc{T}_{(1,j)}$ minimises the number of disjoint pairs over all families $\mc{F} \subseteq S_n$ of size $k(n-1)!$.
\end{prop}
\begin{proof}
Let $\mc{F}\subseteq S_n$ be an extremal family of size $k(n-1)!$, and let $f=\1_{\mc{F}}$. By~\eqref{est-T}, we must have $\disj(\mc F) \le \disj(\mc T) = \binom{k}{2} (n-1)! D_{n-1}$.  Hence, as in the proof of Lemma~\ref{lem-removal-perm}, we can use~\eqref{eq-f2} with $\ep=\be=0$, and so $\|f_2\|^2=0$. It follows from~\eqref{eq-dp-l} that $\disj(\mc{F}) \ge \binom{k}{2}(n-1)!D_{n-1} = \disj(\mc T)$, showing that $\mc T$ minimises the number of disjoint pairs.
\end{proof}


\subsection{Intersection graphs}

The removal lemma states that families with relatively few disjoint pairs must be close to unions of cosets.  While this describes their large-scale structure, it falls short of determining the finer details of such families.  As we have observed previously, certain pairs of cosets are disjoint, while other pairs share a small number of permutations.  In order to keep track of this information, we introduce the notion of an intersection graph.

Given a union $\mc G = \mc T_{(i_1, j_1)} \cup \hdots \cup \mc T_{(i_k,j_k)}$ of $k$ different cosets in $S_n$, its \emph{intersecting graph} is the graph with vertex set $\{(i_1, j_1), \hdots, (i_k, j_k)\}$ and edges between pairs corresponding to cosets with non-empty intersection.  As remarked before Definition~\ref{def:lexperm}, we therefore have $(i,j) \sim (i',j')$ if and only if $i \neq i'$ and $j \neq j'$; that is, when these vertices do not lie on an axis-aligned line in $\mathbb{Z}^2$.

For Theorem~\ref{thm-supsat-perm}, we need to show that pairwise disjoint cosets minimise the number of disjoint pairs.  To that end, we call a union $\mc G$ of $k$ cosets \emph{canonical} if at least $k - 1$ of its cosets are pairwise disjoint.  In terms of the intersection graph $G$ of $\mc G$, this means there is an axis-aligned line containing at least $v(G) - 1$ vertices.  For example, when $s = k (n-1)!$, the lexicographic family $\mc T(n,s)$ is canonical, as all the vertices $(i,j)$ of its intersection graph lie on the line $i = 1$.

Our next proposition, central to the proof of Theorem~\ref{thm-supsat-perm}, describes how the intersection graph of a union $\mc G$ of cosets can be used to bound the size of $\mc G$ and the number of disjoint pairs it is involved in.  For this we require some further notation.  Given a graph $G$ and an integer $t \ge 1$, we denote by $k_t(G)$ the number of $t$-cliques in $G$.  In particular, we have $k_1(G) = v(G)$ and $k_2(G) = e(G)$.  When the graph $G$ is clear from context, we write $k_t$ for $k_t(G)$.

\begin{prop} \label{prop:int-graphs}
There is some $c > 0$ such that if, for $2 \le k_1 \le c n^{1/2}$, $\mc G$ is the union of $k_1$ cosets in $S_n$ with intersection graph $G$, then the following hold:
\begin{enumerate}
	\item[(a)] $\card{\mc G} = k_1 (n-1)! - k_2(n-2)! + k_3(n-3)! \pm k_4 (n-4)!$,
	\item[(b)] $\disj(\pi,\mc G) = k_1 D_{n-1} - k_2 D_{n-2} \pm 3 k_1 k_2 (n-3)!$ for every $\pi \in S_n \setminus \mc G$, 
	\item[(c)] $\disj(\mc G) = \binom{k_1}{2} (n-1)! D_{n-1} - (k_1 - 1) k_2 (n-2)! D_{n-1} \pm 2 k_1^2 k_2 (n-1)!(n-3)!$, and
	\item[(d)] $\disj(\mc G) \ge \disj(\mc T(n,\card{\mc G}))$, with equality if and only if $\mc G$ is canonical.
\end{enumerate}
\end{prop}

The proof of Proposition~\ref{prop:int-graphs}, though elementary, is rather technical, involving careful and repeated application of the Bonferroni inequalities to estimate the number of permutations in a union of cosets that are disjoint from a given permutation.  We therefore defer the proof to Appendix~\ref{app:int-graphs}, and instead proceed to show how the proposition can be combined with Lemma~\ref{lem-removal-perm} to prove Theorem~\ref{thm-supsat-perm}.

\subsection{Supersaturation} \label{sec:supsat-perm}
Here we prove Theorem~\ref{thm-supsat-perm}. Our strategy is to use Lemma~\ref{lem-removal-perm} to reduce the statement to the case when $\mc{F}$ is a union of some cosets in $S_n$, and then apply Proposition~\ref{prop:int-graphs} to obtain the desired lower bound on the number of disjoint pairs.

\begin{proof}[Proof of Theorem~\ref{thm-supsat-perm}]
Let $c_{\ref{lem-removal-perm}}$ and $C_{\ref{lem-removal-perm}}$ be the positive constants from Lemma~\ref{lem-removal-perm}, and set
\[
c=\min\left\{ c_{\ref{lem-removal-perm}}, \ 10^{-5}C_{\ref{lem-removal-perm}}^{-2},\ 10^{-2}\right\}.
\]

Now letting $n, k$ and $\ep$ be as in the statement of the theorem, let $\mc{F} \subseteq S_n$ be an extremal family of $s=(k+\ep)(n-1)!$ permutations. In the first part of our proof, we establish Claim~\ref{cl-rs}, a rough structural result for $\mc{F}$.

\begin{claim}\label{cl-rs}
Either $\mc{F}$ contains $k$ cosets or $\mc{F}$ is contained in a union of $k$ cosets.
\end{claim}
\begin{proof}
Since $|\ep|\le ck^{-3} \le c_{\ref{lem-removal-perm}}$ and  $\disj(\mc{F}) \le \disj(\mc{T}(n,s))$ by the extremality of $\mc{F}$, we may apply Lemma~\ref{lem-removal-perm} to $\mc{F}$ with $\be_{\ref{lem-removal-perm}}=0$ to find a union $\mc{G} = \bigcup_{i=1}^k \mc T_i$ of $k$ cosets in $S_n$ such that 
\begin{eqnarray}\label{eq-sd}
	\card{\mc{F} \De \mc{G}} \le  C_{\ref{lem-removal-perm}}k^2\left(\frac{1}{n}+\sqrt{\frac{6|\ep|}{k}}\right)(n-1)!.
\end{eqnarray}

Let $\mc{A}=\mc{F}\setminus \mc{G}$ and $\mc{B}=\mc{G}\setminus\mc{F}$. We may assume that $\mc{A}\neq \emptyset$ and $\mc{B}\neq \emptyset$, otherwise either $\mc{F}\subseteq \mc{G}$ or $\mc{G}\subseteq \mc{F}$ as claimed. We shall show that if the permutations in $\mc{A}$ are replaced by those in $\mc{B}$, the number of disjoint pairs decreases, which then contradicts the extremality of $\mc{F}$. Fix two arbitrary permutations $\si\in\mc{A}$ and $\pi\in\mc{B}$. It suffices to show that $\disj(\si,\mc{F})>\disj(\pi,\mc{F})$.

First, using~\eqref{eq-sd}, $|\ep|\le ck^{-3}$, $k \le c n^{1/2}$ and $c \le 10^{-5}C_{\ref{lem-removal-perm}}^{-2}$, we see that
\begin{eqnarray}\label{eq-ABs}
	\card{\mc{A}}+\card{\mc{B}}= \card{\mc{F}\De\mc{G}}\le 0.02(n-1)!.
\end{eqnarray}
 Recall that any two cosets in $S_n$ have at most $(n-2)!$ elements in common, and that a permutation is disjoint from $D_{n-1}$ other permutations in any coset not containing it.  Since $D_{n-1}=d_{n-1}+d_{n-2}=(e^{-1}+o(1))(n-1)!$ and $k \le cn^{1/2}\le 10^{-2}n^{1/2}$, we have
\begin{align*}
	\disj(\si,\mc{F})&\ge \disj(\si,\mc{G}\setminus\mc{B})\ge \disj(\si,\mc{G})-\card{\mc{B}}\ge \sum_{i = 1}^{k}\disj(\si,\mc{T}_i)-\sum_{i<j}\card{\mc{T}_{i}\cap \mc{T}_j}-\card{\mc{B}}\\
	&\stackrel{\eqref{eq-ABs}}{\ge} kD_{n-1}-{k\choose 2}(n-2)!-0.02(n-1)!>(k-0.1)D_{n-1}.
\end{align*}
On the other hand, $\pi$ is contained in $\mc G$, and thus can have disjoint pairs to at most $k-1$ of the cosets in $\mc G$.  Hence,
\begin{align*}
\disj(\pi,\mc{F})&=\disj(\pi,\mc{F}\cap \mc{G})+\disj(\pi,\mc{F}\setminus\mc{G})=\disj(\pi,\mc{G}\setminus \mc{B})+\disj(\pi,\mc{A})\\
&\le (k-1)D_{n-1}+\card{\mc{A}}\stackrel{(\ref{eq-ABs})}{<}(k-0.2)D_{n-1}<\disj(\si,\mc{F}).\qedhere
\end{align*}
\end{proof}	

Next, we combine this claim with Proposition~\ref{prop:int-graphs} to bound $\disj(\mc F)$ from below and finish the proof. We consider two cases, depending on the sign of $\ep$.

\noindent {\bf Case 1:} $\ep\le 0$. We have shown in Claim~\ref{cl-rs} that either $\mc{F} \supseteq\mc{G}$ or $\mc{F} \subseteq\mc{G}$, where $\mc{G}$ is a union of some $k$ cosets. Let $t=\card{\mc{G}}$.

We first treat the case $\mc{F} \supseteq\mc{G}$. Observe that since $t \le s$, $\mc T(n,t) \subseteq \mc T(n,s)$, and from Proposition~\ref{prop:int-graphs}(d) we have $\disj(\mc{G})\ge \disj(\mc{T}(n,t))$. Since $s=(k+\ep)(n-1)! \le k(n-1)!$, the family $\mc{T}(n,s)$ is contained in a union of $k$ disjoint cosets in $S_n$. 
Hence $\disj(\pi,\mc{T}(n,s)) \le (k-1)D_{n-1}$ for every $\pi\in\mc{T}(n,s) \setminus \mc{T}(n,t)$, and there are  $\card{\mc{F}\setminus \mc{G}}$ such permutations $\pi$.

Moreover, as $\mc{G}$ is a union of $k$ cosets, we have
\[
\disj(\si,\mc{G}) \ge kD_{n-1}-{k \choose 2} (n-2)!>(k-0.5)D_{n-1}
\] 
for each $\si \in\mc{F}\setminus\mc{G}$, and there are again $\card{\mc{F}\setminus \mc{G}}$ such permutations $\si$. Altogether, we deduce that, as required,
\[ \disj(\mc{F}) \ge \disj(\mc{G}) + \sum_{\si \in \mc F \setminus \mc G} \disj(\si, \mc G) > \disj(\mc T(n,t)) + \sum_{\pi \in \mc T(n,s) \setminus \mc T(n,t)} \disj(\pi, \mc T(n,s)) \ge \disj(\mc{T}(n,s)). \]

We next deal with the case $\mc{F} \subseteq \mc{G}$. It is convenient to think of $\mc{F}$ as a family obtained by removing permutations in $\mc{G}$ one by one. Since $\mc{G}$ is a union of $k$ cosets in $S_n$, the number of disjoint pairs is decreased by at most $(k-1)D_{n-1}$ each time. Following the same process for the family $\mc{T}(n,t)$, we see that the number of disjoint pairs is decreased by exactly $(k-1)D_{n-1}$ each time we remove a permutation from the last coset in $\mc{T}(n,t)$. Moreover, at the beginning of the process, $\disj(\mc{G})\ge \disj(\mc{T}(n,t))$ by Proposition~\ref{prop:int-graphs}(d). Thus 
\[
\disj(\mc{F})\ge \disj(\mc{G})-(t-s)(k-1)D_{n-1}\ge \disj(\mc{T}(n,t))-(t-s)(k-1)D_{n-1}=\disj(\mc{T}(n,s)),
\]
completing the proof in Case 1.

\noindent {\bf Case 2:} $\ep > 0$. This case will be handled rather differently. Since $\ep>0$, formula \eqref{est-T} gives 
\begin{equation}\label{positive-T}
\disj(\mc{T}(n,s))=\left(\binom{k}{2}+k\ep\right)(n-1)!D_{n-1}.
\end{equation}
Also, as $\card{\mc{F}}=(k+\ep)(n-1)!>k(n-1)!$, Claim~\ref{cl-rs} shows that $\mc{F}=\mc{G}\sqcup \mc{H}$, where $\mc{G}$ is a union of $k$ (not necessarily disjoint) cosets in $S_n$.

If $\mc{G}$ is a union of $k$ disjoint cosets, then
\[
\disj(\mc{F})=\disj(\mc{G})+\disj(\mc{H},\mc{G})+\disj(\mc{H}) \ge \disj(\mc{G})+\disj(\mc{H},\mc{G})\\
=\disj(\mc{T}(n,s)),
\]
where equality holds if and only if $\disj(\mc{H})=0$, that is, $\mc{H}$ is intersecting.

It remains to verify that $\disj(\mc{F})\ge \disj(\mc{T}(n,s))$ when the $k$ cosets of $\mc{G}$ are not pairwise disjoint. In this scenario we in fact have a strict inequality. Indeed, let $G$ be the intersection graph of $\mc{G}$. We shall use the inequality $\disj(\mc{F})\ge \disj(\mc{G})+\disj(\mc{H},\mc{G})$ to lower bound $\disj(\mc{F})$. By Proposition~\ref{prop:int-graphs}(c), we have
\begin{equation}\label{positive-G}
\disj(\mc{G})=\binom{k}{2}(n-1)!D_{n-1}-(k-1)k_2(n-2)!D_{n-1}\pm 2k^2k_2 (n-1)!(n-3)!.
\end{equation} 

We next estimate the number of disjoint pairs between $\mc{H}$ and $\mc{G}$. By Proposition~\ref{prop:int-graphs}(b), 
\[\disj(\pi,\mc{G})=kD_{n-1}-k_2 D_{n-2} \pm 3kk_2 (n-3)!\]
for every $\pi \in \mc{H}$.
Furthermore, using Proposition~\ref{prop:int-graphs}(a) to estimate $\card{\mc G}$ gives
\[ \card{\mc{H}}=\card{\mc{F}}-\card{\mc{G}}=\ep (n-1)!+k_2 (n-2)!\pm 2kk_2 (n-3)!.\] 
Therefore, noting that $(n-1)!D_{n-2}=(n-2)!D_{n-1}\pm (n-1)!(n-3)!$, we get
\begin{align}\label{positive-H-G}
\notag\disj(\mc{H},\mc{G})&= \left(kD_{n-1}-k_2 D_{n-2} \pm  3kk_2 (n-3)!\right) \card{\mc{H}}\\
& =k\ep(n-1)!D_{n-1}+(k-\ep)k_2(n-2)!D_{n-1}\pm 3k^2k_2 (n-1)!(n-3)!.
\end{align}
Combining \eqref{positive-T}, \eqref{positive-G} and \eqref{positive-H-G}, and simplifying gives
\[
\disj(\mc{G})+\disj(\mc{H},\mc{G})-\disj(\mc{T}(n,s)) \ge (1-\ep)k_2 (n-2)!D_{n-1} -5k^2k_2 (n-1)!(n-3)!>0,
\]
since $k_2\ge 1$, $k\le cn^{1/2}\le 10^{-2}n^{1/2}$ and $D_{n-1} \ge (n-1)!/3$. Thus $\disj(\mc{F})\ge \disj(\mc{G})+\disj(\mc{H},\mc{G})>\disj(\mc{T}(n,s))$, completing the proof of Theorem \ref{thm-supsat-perm}.
\end{proof}



\section{Supersaturation for uniform set systems}\label{sec:uniform}
In this section, we shall prove Theorem~\ref{thm:setsupersat}, but first let us examine $\disj(n,k,s)$.  When $\bnk - \binom{n-r+1}{k} \le s \le \bnk - \binom{n-r}{k}$, if we write $s = \bnk - \binom{n-r+1}{k} + \ga \binom{n-r}{k-1}$ where $\ga \in [0,1]$, $\mc{L}(n,k,s)$ consists of the full stars with centres in $[r-1]$, with a further $\ga \binom{n-r}{k-1}$ sets from the star with centre $r$.  Let $\mc{L}(i) = \{ L \in \mc{L}(n,k,s) : i \in L \}$ and $\mc{L}^*(i) = \{ L \in \mc{L}(n,k,s): \min L = i \}$.  One can then compute the number of disjoint pairs as
\begin{equation} \label{eqn:setdisjexact}
\disj(n,k,s) = \sum_{i=1}^{r-1} \disj( \cup_{j > i} \mc{L}^*(j), \mc{L}^*(i) ) = \sum_{i = 1}^{r-1} \left( s - \left( \bnk - \binom{n-i}{k} \right) \right) \binom{n-i-k}{k-1}.
\end{equation}
This expression is quite unwieldy, so we shall make use of a few estimates.  We first note that any set outside a star has exactly $\binom{n-k-1}{k-1}$ disjoint pairs with the star, so
\begin{align} \label{ineq:setbound1}
\disj(n,k,s) &\le \sum_{1 \le i <  j \le r-1} \disj(\mc{L}(i), \mc{L}(j)) + \sum_{1 \le i \le r-1} \disj(\mc{L}(i), \mc{L}^*(r)) \notag \\
	&\le \left( \binom{r-1}{2} + (r-1)\ga \right) \bekr \binom{n-k-1}{k-1}.
\end{align}
This is only an upper bound as we overcount disjoint pairs involving sets belonging to multiple stars.  For an even simpler upper bound, observe that every set belongs to at least one of the $r$ stars, and is not disjoint from any other set in its star.  In the worst case, there are an equal number of sets in each star, with each set disjoint from at most a $\left(1 - \frac{1}{r}\right)$-proportion of the family.  We thus have
\begin{equation} \label{ineq:setbound2}
\disj(n,k,s) \le \frac12 \left( 1 - \frac{1}{r} \right) s^2.
\end{equation}

We shall use these upper bounds on the number of disjoint pairs present in any extremal family.

\subsection{Tools}

There are two main tools we use in our proof of Theorem~\ref{thm:setsupersat}: a removal lemma for disjoint pairs, and the expander-mixing lemma applied to the Kneser graph.  Before proving the theorem, we introduce these tools and explain how we shall use them.

\subsubsection{Removal lemma}

Using a result of Filmus~\cite{Filmus}, Das and Tran~\cite[Theorem 1.2]{DT} proved the following removal lemma, showing that large families with few disjoint pairs must be close to a union of stars.

\begin{lemma}[Das and Tran] \label{lem:setremoval}
There is an absolute constant $C > 1$ such that if $n, k$ and $\l$ are positive integers satisfying $n > 2k \l^2$, and $\mc{F} \subset \binom{[n]}{k}$ is a family of size $\card{\mc{F}} = (\l - \al) \bekr$ with at most $\left( \binom{\l}{2} + \be \right) \bekr \binom{n-k-1}{k-1}$ disjoint pairs, where $\max \{ 2 \l \card{\al}, \card{\be} \} \le \frac{n-2k}{(20C)^2 n}$, then there is a family $\mc{S}$ that is the union of $\l$ stars satisfying $$\card{\mc{F} \De \mc{S}} \le C \left( (2 \l - 1) \al + 2 \be \right) \frac{n}{n-2k} \bekr.$$
\end{lemma}

Observe that the bound on the number of disjoint pairs in the lemma is very similar to the upper bound given in~\eqref{ineq:setbound1}.  Thus one may interpret this result as a stability version of our previous calculation: any family with size similar to the union of $r-1$ stars without many more disjoint pairs can be made a union of $r-1$ stars by exchanging only a small number of sets.  Given this stability, it is not difficult to show that the lexicographic ordering is optimal in this range.

\begin{cor} \label{cor:setclosetostar}
There is some constant $c > 0$ such that if $r, k$ and $n$ are positive integers satisfying $n \ge 2c^{-1} k^2r^3$, and $s = \bnk - \binom{n-r+1}{k} + \ga \binom{n-r}{k-1}$, where $\ga \in [0 , \frac{c}{r} ]$, any family $\mc{F} \subset \binom{[n]}{k}$ of size $s$ has $\disj(\mc{F}) \ge \disj(n,k,s)$.
\end{cor}

\begin{proof}
Let $C$ be the constant from Lemma~\ref{lem:setremoval} and choose $c = \frac{n-2k}{2(20C)^2 n}$.  For the given range of $s$, the lexicographic initial segment has $r-1$ full stars with one small partial star, so we wish to apply Lemma~\ref{lem:setremoval} with $\l = r-1$.

Let $\mc{F}$ be a subfamily of $\binom{[n]}{k}$ with $s$ sets and the minimum number of disjoint pairs.  Note that $s = \left( \l + \al \right) \bekr$, where $\ga - \frac{kr^2}{2n} \le \al \le \ga$.  In particular, we have $\card{\al} \le \frac{c}{r}$.  By optimality of $\mc{F}$, and our calculation in~\eqref{ineq:setbound1}, we also have $\disj(\mc{F}) \le \disj(n,k,s) \le \left( \binom{\l}{2} + (r-1)\ga\right) \bekr \binom{n-k-1}{k-1}$, and hence take $\be = (r-1)\ga$.

We thus have $\card{\be} = (r-1) \ga \le c < \frac{n-2k}{(20C)^2 n}$ and $2 \l \card{\al} \le 2c = \frac{n-2k}{(20C)^2 n}$, and hence we may apply Lemma~\ref{lem:setremoval}.  This gives a family $\mc{S}$, a union of $\l$ stars, such that 
\[ \card{ \mc{F} \De \mc{S} } \le C \left( (2 \l - 1) \al + 2 \be \right) \frac{n}{n-2k} \bekr \le 4c C \frac{n}{n-2k} \bekr < \frac{1}{200} \bekr. \]

Hence, we know an optimal family $\mc{F}$ must be close to a union of $\l$ stars $\mc{S}$.  We first show that $\mc{S} \subseteq \mc{F}$.  If not, there is some set $F \in \mc{F} \setminus \mc{S}$ in our family, as well as a set $G \in \mc{S} \setminus \mc{F}$ missing from our family (note that $\card{\mc{F}}\ge \card{\mc{S}}$). For each star in $\mc{S}$, there are at most $\bekr - \binom{n-k-1}{k-1} \le \frac{k^2}{n} \bekr$ sets intersecting $F$, and hence $F$ intersects at most $\frac{\l k^2}{n} \bekr$ sets from $\mc{F} \cap \mc{S}$.  Even if $F$ intersects every set in $\mc{F} \setminus \mc{S}$, it can intersect at most $\left( \frac{\l k^2}{n} + \frac{1}{200} \right) \bekr < \frac12 \bekr$ sets in $\mc{F}$.

On the other hand, the set $G$ is in one of the stars of $\mc{S}$, which contains at least $\bekr - \card{ \mc{S} \setminus \mc{F} } \ge \left( 1 - \frac{1}{200} \right) \bekr > \frac23 \bekr$ sets of $\mc{F}$.  Hence replacing $F$ by $G$ in $\mc{F}$ strictly increases the number of intersecting pairs, thus decreasing the number of disjoint pairs, contradicting the optimality of $\mc{F}$.

Thus we have $\mc{S} \subseteq \mc{F}$. Let $\mc{H} = \mc{F} \setminus \mc{S}$.  We then have
\[ \disj(\mc{F}) = \disj(\mc{S}) + \disj(\mc{S}, \mc{H}) + \disj(\mc{H}). \]
Since every set outside a union of $\l$ stars is contained in exactly the same number of disjoint pairs with sets from the stars, the terms $\disj(\mc{S})$ and $\disj(\mc{S}, \mc{H})$ are determined by $\l$ and $s$, and independent of the structure of $\mc{F}$.  It follows that $\disj(\mc{F})$ is minimised precisely when $\disj(\mc{H})$ is minimised.  As $\card{\mc{H}} = \card{\mc{F}} - \card{\mc{S}} = \ga \binom{n-r}{k-1} \le \binom{n-r}{k-1}$, we may take $\mc{H}$ to be an intersecting family, and so $\disj(\mc{H}) = 0$ is possible.  Since in $\mc{L}(n,k,s)$, the set $\mc{H}$ corresponds to the final (intersecting) partial star, it follows that $\mc{L}(n,k,s)$ is optimal, and so $\disj(\mc{F}) \ge \disj(n,k,s)$ for any family $\mc{F}$ of $s$ sets.
\end{proof}

\subsubsection{Expander-mixing lemma}

The second tool we shall use is the expander-mixing lemma\footnote{This also plays a key role in the proof of Lemma~\ref{lem:setremoval}, and so in some sense is the foundation for this entire proof.} of Alon and Chung~\cite{Alon-C}, which relates the spectral gap of a $d$-regular graph to its edge distribution. Since the graph is $d$-regular, its largest eigenvalue is trivially $d$, corresponding to the constant eigenvector. In what follows, an $(n,d,\lam)$-graph is a $d$-regular $n$-vertex graph whose largest non-trivial eigenvalue (in absolute value) is $\lam$.

\begin{lemma}[Alon and Chung] \label{lem:expandermixing}
Let $G$ be an $(n,d,\lam)$-graph, and let $S, T$ be two vertex subsets.  Then
\[ \card{ e(S,T) - \frac{d\card{S} \card{T}}{n} } \le \lam \sqrt{ \card{S} \card{T} }. \]
\end{lemma}

As we are interested in counting disjoint pairs, we shall apply the expander-mixing lemma to the Kneser graph, where the vertices are sets and edges represent disjoint pairs.  The spectral properties of the Kneser graph were determined by Lov\'asz~\cite{Lovasz}. In particular, the Kneser graph $KG(m,a)$ for $a$-uniform sets over $[m]$ is an $\left(\binom{m}{a}, \binom{m-a}{a}, \binom{m-a-1}{a-1} \right)$-graph.  We shall combine this with Lemma~\ref{lem:expandermixing} to obtain a useful corollary.

\begin{cor} \label{cor:setspectral}
Given $1 \le i < j \le n$, and $k$-uniform families $\mc{F}(i)$ of subsets of $[n]$ containing $i$ and $\mc{F}(j)$ of subsets of $[n]$ containing $j$, 
\[ \disj(\mc{F}(i), \mc{F}(j)) \ge \left(1 - \frac{k^2}{n} \right) \card{\mc{F}(i)}\card{\mc{F}(j)} - \frac{3k}{2n} \left( \card{\mc{F}(i)} + \card{\mc{F}(j)} \right) \bekr. \]
\end{cor}

\begin{proof}
Without loss of generality we assume $i = n-1$ and $j = n$. Let $\mc{A} = \{ F \setminus \{n-1 \} : F \in \mc{F}(n-1), \; n \notin F \}$ and $\mc{B} = \{ F \setminus \{ n \} : F \in \mc{F}(n), \; n-1 \notin F \}$, and observe that $\disj(\mc{F}(n-1), \mc{F}(n)) = \disj(\mc{A}, \mc{B})$.  Furthermore, we have $\mc{A}, \mc{B} \subseteq \binom{[n-2]}{k-1}$, with $\card{\mc{A}} \ge \card{\mc{F}(n-1)} - \binom{n-2}{k-2}$ and $\card{\mc{B}} \ge \card{\mc{F}(n)} - \binom{n-2}{k-2}$.

Since disjoint pairs between $\mc{A}$ and $\mc{B}$ correspond to edges between the corresponding vertex sets in the Kneser graph $KG(n-2,k-1)$, Lemma~\ref{lem:expandermixing} gives
\[ \disj(\mc{F}(n-1), \mc{F}(n)) = \disj(\mc{A},\mc{B}) \ge \frac{\binom{n-k-1}{k-1}}{\binom{n-2}{k-1}} \card{\mc{A}}\card{\mc{B}} - \binom{n-k-2}{k-2} \sqrt{ \card{\mc{A}} \card{\mc{B}}}. \]

We now recall that $\card{\mc{F}(n-1)} - \binom{n-2}{k-2} \le \card{\mc{A}} \le \card{\mc{F}(n-1)}$, with similar bounds holding for $\mc{B}$.  We shall also remove the square root by appealing to the AM-GM inequality.  Also observe that $\binom{n-k-1}{k-1} \ge \left(1 - \frac{k^2}{n} \right) \binom{n-2}{k-1}$ and $\binom{n-k-2}{k-2} \le \frac{k}{n} \bekr$.  Hence $\disj(\mc{F}(n-1), \mc{F}(n))$ is at least
\[ \left(1 - \frac{k^2}{n} \right) \left( \card{\mc{F}(n-1)} - \binom{n-2}{k-2} \right) \left( \card{\mc{F}(n)} - \binom{n-2}{k-2} \right) - \frac{k}{2n} \left( \card{\mc{F}(n-1)} + \card{\mc{F}(n)} \right) \bekr. \]

Noting that $\binom{n-2}{k-2} \le \frac{k}{n} \bekr$, taking the main order term and collecting the negative terms then gives the desired bound.
\end{proof}

\subsection{Proof of Theorem~\ref{thm:setsupersat}}

With the preliminaries in place, we now proceed with the proof of the main theorem.

\begin{proof}[Proof of Theorem~\ref{thm:setsupersat}]
We prove the result by induction on $s$.  For the base case, if $s \le \bnk - \binom{n-1}{k} = \bekr$, then $\mc{L}(n,k,s)$ consists of sets that all contain the element $1$. Hence $\disj(\mc{L}(n,k,s)) = 0$, which is clearly optimal.

For the induction step, we have $s = \bnk - \binom{n-r+1}{k} + \ga \binom{n-r}{k-1}$ for some $r \ge 2$ and $\ga \in (0,1]$.  Letting $c$ be the positive constant from Corollary~\ref{cor:setclosetostar}, if $\ga \in (0, \frac{c}{r}]$, we are done.  Hence we may assume $\ga \in (\frac{c}{r}, 1]$.  Let $\mc{F}$ be a $k$-uniform set family over $[n]$ of size $s$ with the minimum number of disjoint pairs.  In particular, we must have $\disj(\mc{F}) \le \disj(n,k,s)$.

For any set $F \in \mc{F}$, by the induction hypothesis we have $\disj(\mc{F} \setminus \{ F \}) \ge \disj(n,k,s-1)$.  Hence $\disj(\{ F\}, \mc{F}) = \disj(\mc{F}) - \disj(\mc{F} \setminus \{ F \}) \le \disj(n,k,s) - \disj(n,k,s-1)$, where the right-hand side is the number of disjoint pairs involving the last set $L$ added to $\mc{L}(n,k,s)$. The set $L$ is in a star of size $\ga \binom{n-r}{k-1}>\frac{\ga}{2}\bekr$, and hence intersects at least $\frac{\ga}{2} \bekr$ sets in $\mc{L}(n,k,s)$.  Thus it follows that every set $F \in \mc{F}$ must also intersect at least $\frac{\ga}{2} \bekr$ sets in $\mc{F}$.

Now suppose that $\mc{F}$ contains a full star; without loss of generality, assume $\mc{F}(1)$ consists of all $\bekr$ sets containing the element $1$.  Let $\mc{G} = \mc{F} \setminus \mc{F}(1)$.  Since $\mc{F}(1)$ is intersecting, and every set outside $\mc{F}(1)$ has exactly $\binom{n-k-1}{k-1}$ disjoint pairs with sets in $\mc{F}(1)$, we have
\[ \disj(\mc{F}) = \disj(\mc{F}(1), \mc{G}) + \disj(\mc{G}) = \card{\mc{G}} \binom{n-k-1}{k-1} + \disj(\mc{G}). \]

Now $\mc{G}$ is a $k$-uniform set family over $[n] \setminus \{1\}$ of size $s' = s - \bekr$, and so by induction $\disj(\mc{G})$ is minimised by the initial segment of the lexicographic order of size $s'$.  However, adding back the full star $\mc{F}(1)$ gives the initial segment of the lexicographic order of size $s$, and as a result $\disj(\mc{F}) \ge \disj(\mc{L}(n,k,s)) = \disj(n,k,s)$.

Hence we may assume that $\mc{F}$ does not contain any full star.  In particular, this means for any set $F \in \mc{F}$ and element $i \in [n]$, we have the freedom to replace $F$ with some set containing $i$.  We shall use such switching operations to show that $\mc{F}$, like $\mc{L}(n,k,s)$, must have a cover of size $r$, from which the result will easily follow.

Relabel the elements if necessary so that for every $i \in [n]$, $i$ is the vertex of maximum degree in $\mc{F} |_{[n] \setminus [i-1]}$.  Let $\mc{F}^*(i) = \{ F \in \mc{F} : \min F = i \}$ be those sets containing $i$ that do not contain any previous element.  Define
\[ X = \left\{ x \in [n] : \card{\mc{F}^*(x)} \ge \frac{\ga}{4k} \bekr \right\}, \]
and let $\mc{F}_1 = \{ F \in \mc{F} : F \cap X \neq \emptyset \}$ and $\mc{F}_2 = \mc{F} \setminus \mc{F}_1 = \{ F \in \mc{F} : F \cap X = \emptyset \}$.  We shall show that $X$ is a cover for $\mc{F}$ (that is, $\mc{F}_1=\mc{F}$ and $\mc{F}_2 = \emptyset$), but to do so we shall first have to establish a few claims.  The first shows that $X$ cannot be too big.

\begin{claim} \label{clm:Xiskbounded}
$\card{X} \le \frac{4kr}{\ga}$.
\end{claim}

\begin{proof}
Observe that the families $\{ \mc{F}^*(x) : x \in X \}$ partition $\mc{F}_1$.  Hence we have
\[ r \bekr \ge s = \card{\mc{F}} \ge \card{\mc{F}_1} = \sum_{x \in X} \card{\mc{F}^*(x)} \ge \frac{\ga}{4 k} \bekr \card{X}, \]
from which the claim immediately follows.
\end{proof}

The next claim asserts that every set in $\mc{F}$ must intersect many sets in $\mc{F}_1$.

\begin{claim} \label{clm:largeF1int}
Every set $F \in \mc{F}$ intersects at least $\frac{\ga}{4} \bekr$ sets in $\mc{F}_1$.
\end{claim}

\begin{proof}
First observe that any element $i \in [n]$ is contained in fewer than $\frac{\ga}{4k} \bekr$ sets in $\mc{F}_2$.  Indeed, the elements $x \in X$ have all of their sets in $\mc{F}_1$, and hence have $\mc{F}_2$-degree zero.  Thus the $\mc{F}_2$-degree of any element is its degree in $\mc{F}|_{[n] \setminus X}$.  If the element of largest $\mc{F}_2$-degree was contained in at least $\frac{\ga}{4k} \bekr$ sets from $\mc{F}_2$, then it would have been in $X$, giving a contradiction.

Now recall that every set $F \in \mc{F}$ must intersect at least $\frac{\ga}{2} \bekr$ sets in $\mc{F}$.  The number of sets in $\mc{F}_2$ it can intersect is at most
\[ \sum_{i \in F} \card{\mc{F}_2(i)} \le k \cdot \frac{\ga}{4 k} \bekr = \frac{\ga}{4} \bekr. \]
Hence the remaining $\frac{\ga}{4} \bekr$ intersections must come from sets in $\mc{F}_1$.
\end{proof}

The following claim combines our previous results with the expander-mixing corollary to provide much sharper bounds on the size of $X$.

\begin{claim} \label{clm:Xissmall}
$\card{X} \le \frac{8r}{\ga}$.
\end{claim}

\begin{proof}
For every $i \in X$, we shall estimate $\disj(\mc{F}^*(i), \mc{F}_1)$.  Since $\{ \mc{F}^*(x) : x \in X \}$ is a partition of $\mc{F}_1$, we have $\disj(\mc{F}^*(i), \mc{F}_1) = \sum_{j \in X \setminus \{i\}} \disj(\mc{F}^*(i), \mc{F}^*(j))$.  Applying Corollary~\ref{cor:setspectral}, we get
\begin{align*}
	\disj(\mc{F}^*(i), \mc{F}_1) &= \sum_{j \in X \setminus \{i\}} \disj(\mc{F}^*(i), \mc{F}^*(j)) \\
	&\ge \sum_{j \in X \setminus \{i\}} \left[ \left(1 - \frac{k^2}{n} \right) \card{\mc{F}^*(i)}\card{\mc{F}^*(j)} - \frac{3k}{2n} \left( \card{\mc{F}^*(i)} + \card{\mc{F}^*(j)} \right) \bekr \right] \\
	&\ge \left(1 - \frac{k^2}{n} \right) \left( \card{\mc{F}_1} - \card{\mc{F}^*(i)} \right) \card{\mc{F}^*(i)}  - \frac{3k}{2n} \left( \card{X} \card{\mc{F}^*(i)} + \card{\mc{F}_1} \right) \bekr.
\end{align*}

By averaging, some $F \in \mc{F}^*(i)$ is disjoint from at least 
\[ \left(1 - \frac{k^2}{n} \right) \left( \card{\mc{F}_1} - \card{\mc{F}^*(i)} \right) - \frac{3k}{2n} \left( \card{X} + \frac{\card{\mc{F}_1}}{\card{\mc{F}^*(i)}} \right) \bekr\]
sets in $\mc{F}_1$.  By Claim~\ref{clm:Xiskbounded}, $\card{X} \le \frac{4kr}{\ga}$.  Since $\card{\mc{F}_1} \le s \le r \bekr$, and $\card{\mc{F}^*(i)} \ge \frac{\ga}{4k}\bekr$, we can lower bound this expression by
\begin{align*}
	\disj(\{ F \}, \mc{F}_1 ) &\ge \left(1 - \frac{k^2}{n} \right) \left( \card{\mc{F}_1} - \card{\mc{F}^*(i)} \right) - \frac{12 k^2 r}{\ga n} \bekr.
\end{align*}

Recalling that $\ga \ge \frac{c}{r}$, we find that $F$ intersects at most
\[ \card{\mc{F}_1} - \disj(\{F\}, \mc{F}_1) \le \card{\mc{F}^*(i)} + \frac{k^2}{n} \card{\mc{F}_1} + \frac{12k^2r}{\ga n} \bekr \le \card{\mc{F}^*(i)} + \frac{13 k^2 r^2}{cn} \bekr \]
sets from $\mc{F}_1$.  By Claim~\ref{clm:largeF1int}, this quantity must be at least $\frac{\ga}{4} \bekr$, which gives
\[ \card{\mc{F}^*(i)} \ge \left( \frac{\ga}{4} - \frac{13 k^2 r^2}{cn} \right) \bekr \ge \frac{\ga}{8} \bekr, \]
since $n > C k^2 r^3$ for some large enough constant $C$.

Hence for every $i \in X$, we in fact have the much stronger bound $\card{\mc{F}^*(i)} \ge \frac{\ga}{8} \bekr$.  Repeating the calculation of Claim~\ref{clm:Xiskbounded} with this new bound gives $\card{X} \le \frac{8r}{\ga}$, as required.
\end{proof}

Our next claim shows that $X$ is indeed a cover for $\mc{F}$.

\begin{claim} \label{clm:Xisacover}
$X$ is a cover for $\mc{F}$; that is, $\mc{F}_1 = \mc{F}$ and $\mc{F}_2 = \emptyset$.
\end{claim}

\begin{proof}
Suppose for contradiction we had some set $F \in \mc{F}_2$.  By Claim~\ref{clm:largeF1int}, at least $\frac{\ga}{4} \bekr$ sets in $\mc{F}_1$ must intersect $F$.  However, each such set must contain at least one element of $X$, which by Claim~\ref{clm:Xissmall} has size at most $\frac{8r}{\ga}$, together with one element from $F$.  Hence there are at most $k \card{X} \binom{n-2}{k-2} \le \frac{8k^2r}{\ga n} \bekr$ sets in $\mc{F}_1$ intersecting $F$. Since $\ga \ge \frac{c}{r}$ and $n > C k^2 r^3$ for some large enough constant $C$, this is less than $\frac{\ga}{4} \bekr$, giving the desired contradiction.
\end{proof}

Now observe that every set in $\mc{F}^*(i)$ meets $X$ in the element $i$.  If it intersects $X$ in further elements, there are at most $\card{X} \le \frac{8r}{\ga}$ choices from the other element, and at most $\binom{n-2}{k-2} \le \frac{k}{n} \bekr$ choices for the rest of the set.  Hence at most $\frac{8kr}{\ga n} \bekr < \frac{\ga}{8} \bekr \le \card{\mc{F}^*(i)}$ sets in $\mc{F}^*(i)$ meet $X$ in at least two elements, and thus there must be some set $F_i \in \mc{F}^*(i)$ such that $F_i \cap X = \{ i \}$.  We shall use this fact to establish the following claim.

\begin{claim} \label{clm:balancedcover}
For all $i, j \in X$, $\card{ \mc{F}^*(j) } - \card{ \mc{F}^*(i)} \le \frac{8k^2r}{\ga n} \bekr$.
\end{claim}

\begin{proof}
Suppose for contradiction $\card{\mc{F}^*(j)} > \card{\mc{F}^*(i)} + \frac{8k^2 r}{\ga n} \bekr$.  Let $F_i \in \mc{F}^*(i)$ be such that $F_i \cap X = \{ i \}$.  Then $F_i$ intersects only the sets that contain $i$ together with sets containing some other element in $X$ and some element in $F_i$.  This gives a total of at most
\[ \card{\mc{F}^*(i)} + k \card{X} \binom{n-2}{k-2} \le \card{\mc{F}^*(i)} + \frac{8k^2r}{\ga n} \bekr < \card{\mc{F}^*(j)}\]
sets.  On the other hand, if we replace $F_i$ by some set $G$ containing $j$ (which we may do, since we assume the family $\mc{F}(j)$ is not a full star), we would gain at least $\card{\mc{F}^*(j)}$ intersecting pairs.  Hence $\mc{F} \cup \{ G \} \setminus \{ F_i \}$ is a family of $s$ sets with strictly fewer disjoint pairs, contradicting the optimality of $\mc{F}$.
\end{proof}

This claim shows that the sets in $\mc{F}$ are roughly equally distributed over the families $\mc{F}^*(i)$, $i \in X$.  To simplify the notation, we let $m = \card{X}$, and so we have $X = [m]$.  By Claim~\ref{clm:Xissmall}, $m \le \frac{8r}{\ga}$.  We shall now proceed to lower-bound the number of disjoint pairs in $\mc{F}$.  Note that $\disj(\mc{F}) = \sum_{1 \le i < j \le m} \disj(\mc{F}^*(i), \mc{F}^*(j))$.  We shall use Corollary~\ref{cor:setspectral} to bound these summands.  We let $s_i = \card{\mc{F}^*(i)} \bekr^{-1}$ and set $\overline{s} = s \bekr^{-1} = \sum_i s_i$.  Note that $s = \bnk - \binom{n-r+1}{k} + \ga \binom{n-r}{k-1}$ and $\ga \in \left[\frac{c}{r},1\right]$ implies $r-1 \le r-1 + \ga - \frac{kr^2}{2n} \le \overline{s} \le r$.

We then have
\begin{align*}
	\disj(\mc{F}) &= \sum_{1 \le i < j \le m} \disj(\mc{F}^*(i), \mc{F}^*(j)) \\
		&\ge \sum_{1 \le i < j \le m} \left[ \left(1 - \frac{k^2}{n} \right) \card{\mc{F}^*(i)}\card{\mc{F}^*(j)} - \frac{3k}{2n} \left( \card{\mc{F}^*(i)} + \card{\mc{F}^*(j)} \right) \bekr \right] \\
		&\ge \left[ \left( 1 - \frac{k^2}{n} \right) \sum_{i < j} s_i s_j - \frac{3k}{2n} \sum_{i < j} (s_i + s_j) \right] \bekr^2 \\
		&\ge \left[ \frac12 \left( 1 - \frac{k^2}{n} \right) \left( \overline{s}^2 - \sum_i s_i^2 \right) - \frac{3km\overline{s}}{2n} \right] \bekr^2 \\
		&\ge \frac12 \left[ \overline{s}^2 - \frac{k^2\overline{s}^2}{n} - \sum_i s_i^2 - \frac{3km \overline{s}}{n} \right] \bekr^2.
\end{align*}

Since $\sum_i s_i = \overline{s}$, there must be some $\l$ with $s_\l \le \frac{\overline{s}}{m}$, and Claim~\ref{clm:balancedcover} then implies that for every $i$, $s_i \le \frac{\overline{s}}{m} + \frac{8k^2r}{\ga n}$.  Hence $\sum_i s_i^2 \le \left( \max_i s_i \right) \sum_i s_i \le \left( \frac{\overline{s}}{m} + \frac{8k^2 r}{\ga n} \right) \overline{s}$, giving
\begin{equation} \label{ineq:finallowerbound}
\disj(\mc{F}) \ge \frac12 \left( 1 - \frac{1}{m} - \frac{k^2}{n} - \frac{8k^2 r}{\ga \overline{s} n} - \frac{3km}{\overline{s} n} \right) \left( \overline{s} \bekr \right)^2.
\end{equation}

\begin{claim} \label{clm:rcover}
$\card{X} = r$; that is, $\mc{F}$ has a cover of size $r$.
\end{claim}

\begin{proof}
Since $s > \bnk - \binom{n-r+1}{k}$, $\mc{F}$ cannot be covered by $r-1$ elements.  Hence we must have $m = \card{X} \ge r$.  

Now recall we have $s = \overline{s} \bekr$, $\overline{s} \ge r-1$, $\ga \ge \frac{c}{r}$, $m \le \frac{8 r}{\ga} \le 8 c^{-1} r^2$ and $n \ge C k^2 r^3$ for some sufficiently large constant $C$.  Substituting these bounds into~\eqref{ineq:finallowerbound}, we find
\[ \disj(\mc{F}) > \frac12 \left( 1 - \frac{1}{m} - \frac{1}{r(r+1)} \right) s^2. \]

However, by~\eqref{ineq:setbound2}, we must have
\[ \disj(\mc{F}) \le \frac12 \left( 1 - \frac{1}{r} \right) s^2. \]

These two bounds together imply $\frac{1}{m} + \frac{1}{r(r+1)} > \frac{1}{r}$, which in turn gives $m < r+1$.  This shows $m = r$, and $X$ is thus a cover of size $r$.
\end{proof}

Hence it follows that $\mc{F}$ is covered by some $r$ elements, which we may without loss of generality assume to be $[r]$.  We now finish with a similar argument as in the proof of Corollary~\ref{cor:setclosetostar}: let $\mc{S}$ be the union of the $r$ stars with centres in $[r]$, and let $\mc{G} = \mc{S} \setminus \mc{F}$ be the missing sets.  Then $\disj(\mc{F}) = \disj(\mc{S}) - \disj(\mc{G}, \mc{S}) + \disj(\mc{G})$ is minimised when $\mc{G}$ is an intersecting family of sets that each meet $[r]$ in precisely one element, which is the case for $\mc{F} = \mc{L}(n,k,s)$.  Hence $\disj(\mc{F})\ge \disj(n,k,s)$, completing the proof of the theorem.
\end{proof}

The problem of minimising the number of disjoint pairs can be viewed as an isoperimetric inequality in the Kneser graph.  The following lemma links isoperimetric problems for small and large families (see, for instance,~\cite[Lemma 2.3]{D-G-S-2}).

\begin{lemma}\label{lem:small-large}
	Let $G=(V,E)$ be a regular graph on $n$ vertices. Then $S\subset V$ minimises the number of edges $e(S)$ over all sets of $\card{S}$ vertices if and only if $V\setminus S$ minimises the number of edges over all sets of $n-\card{S}$ vertices.	
\end{lemma}

The following corollary, which is a direct consequence of Theorem \ref{thm:setsupersat} and Lemma \ref{lem:small-large}, shows that the complements of the lexicographical initial segments, which are isomorphic to initial segments of the colexicographical order, are optimal when $s$ is close to $\bnk$.

\begin{cor}
	There exists a positive constant $C$ such that the following statement holds. Provided $n \ge C k^2 r^3$ and $\binom{n-r}{k}\le s\le \bnk$, $\binom{[n]}{k}\setminus\mc{L}(n,k,\bnk-s)$ minimises the number of disjoint pairs among all systems of $s$ sets in $\binom{[n]}{k}$.	
\end{cor}


\section{Typical structure of set systems with given matching number}\label{sec:aa}

\subsection{Families with no matching of size $s$}

In this section we describe the structure of $k$-uniform set families without matchings of size $s$.  The following lemma, which follows readily from \cite[Lemmas 2.2 and 2.3]{BDDLS}, gives a sufficient condition for the trivial extremal families to be typical.

\begin{lemma}\label{lem:count-fam-decreasing}
Let $\mathbf{P}$ be a decreasing property. Let $N_0$ denote the size of the extremal (that is, largest) family with property $\mathbf{P}$, $N_1$ the size of the largest non-extremal maximal family, and suppose two distinct extremal families have at most $N_2$ members in common. Suppose further that the number of extremal families is $T$, and there are at most $M$ maximal families. Provided
\begin{equation}\label{count-fam-decreasing}
2\log M + \max(N_1,N_2)-N_0 \rightarrow -\infty,
\end{equation} 	
the number of families with property $\mathbf{P}$ is $(T+o(1))2^{N_0}$.
\end{lemma}

We will apply Lemma \ref{lem:count-fam-decreasing} with $\mathbf{P}$ being the property of avoiding a matching of size $s$ or, equivalently, of not containing $s$ pairwise disjoint sets.  To do so, we first bound the number of maximal families with no matching of size $s$.
\begin{prop}\label{prop:max-no-s-disj}
	The number of maximal $k$-uniform families over $[n]$ with no matching of size $s$ is at most $\bnk^{\binom{sk}{k}}$.
\end{prop}
\begin{proof}
Given $\mc{F}\subset \binom{[n]}{k}$, let $\mc{I}(\mc{F})=\big\{G\in \binom{[n]}{k}: \text{$\mc{F}\cup\{G\}$ does not have $s$ pairwise disjoint sets}\}$. Note that $\mc{F}$ does not contain a matching of size $s$ if and only if $\mc{F}\subset\mc{I}(\mc{F})$, while $\mc{F}$ is maximal if and only if $\mc{I}(\mc{F})=\mc{F}$.
Given a maximal family $\mc{F}$, we say that $\mc{G}\subset \mc{F}$ is a {\em generating family} of $\mc{F}$ if $\mc{I}(\mc{G})=\mc{F}$.

Let $\mc{F}_0=\{F_1,\ldots,F_m\}\subset\mc{F}$ be a minimal generating family of $\mc{F}$. By the minimality of $\mc{F}_0$, we must have $\mc{I}(\mc{F}_0\setminus\{F_i\})\supsetneq \mc{F}=\mc{I}(\mc{F}_0)$, for each $1\le i\le m$. Hence we can find some set $G_{i,s-1} \in \mc{I}(\mc{F}_0\setminus\{F_i\})\setminus\mc{I}(\mc{F}_0)$. It follows that there exist $s-2$ sets $G_{i,1},\ldots,G_{i,s-2}$ in $\mc{F}_0$ such that $F_i,G_{i,1},\ldots,G_{i,s-2}$ and $G_{i,s-1}$ are pairwise disjoint, while for every $j\ne i$, $F_j,G_{i,1},\ldots,G_{i,s-2}$ and $G_{i,s-1}$ are not pairwise disjoint. In other words, if we let $G_i=G_{i,1}\cup\ldots\cup G_{i,s-1}$, then $F_i\cap G_i=\emptyset$ and $F_j\cap G_i \ne \emptyset$ for $j\ne i$. Given these conditions, we may apply the Bollob\'as set-pairs inequality \cite{Bollobas65} to bound the size of $\mc{F}_0$.
\begin{theorem}[Bollob\'as]\label{bollobas:set-pairs}
	Let $A_1,\ldots,A_m$ be sets of size $a$ and $B_1,\ldots,B_m$ sets of size $b$ such that $A_i\cap B_i=\emptyset$ and $A_i\cap B_j\ne \emptyset$ for every $i\ne j$. Then $m\le \binom{a+b}{a}$.
\end{theorem}

	We apply this to the pairs $\{(A_i,B_i)\}_{i=1}^m$, where for $1\le i\le m$ we take $A_i=F_i$ and $B_i=G_i$. The conditions of Theorem \ref{bollobas:set-pairs} are satisfied, and hence we deduce $m\le \binom{sk}{k}$.
	
	We map each maximal family $\mc{F}$ to a minimal generating family $\mc{F}_0\subset\mc{F}$. This map is injective because $\mc{I}(\mc{F}_0)=\mc{F}$. We have shown that $\card{\mc{F}_0}\le \binom{sk}{k}$, and thus the number of maximal families is bounded from above by $\sum_{i=0}^{\binom{sk}{k}}\binom{\bnk}{i} \le \bnk^{\binom{sk}{k}}$,
	as desired.
\end{proof}

\begin{proof}[Proof of Theorem \ref{thm:no-s-disj}]
We shall verify that the condition \eqref{count-fam-decreasing} from Lemma \ref{lem:count-fam-decreasing} holds. A result of Frankl \cite[Theorem 1.1]{Frankl13} states that when $n\ge (2s-1)k-s+1$, the extremal families with no $s$ pairwise disjoint sets are isomorphic to $\left\{F\in \binom{[n]}{k}:F\cap [s-1]\ne \emptyset\right\}$, and consequently we may take $N_0=\bnk-\binom{n-s+1}{k}$ and $T=\binom{n}{s-1}$. Moreover, it is not difficult to see that the intersection of any two extremal families has size at most $N_2=\bnk-\binom{n-s+1}{k}-\binom{n-s}{k-1}$. Furthermore, a result due to Frankl and Kupavskii \cite[Theorem 5]{fk17-stability} implies that
$N_1\le \bnk-\binom{n-s+1}{k}-\frac{1}{s+1}\binom{n-k-s+1}{k-1}$  for $n\ge 2sk-s$. Hence
$\max(N_1,N_2)\le \bnk-\binom{n-s+1}{k}-\frac{1}{s+1}\binom{n-k-s+1}{k-1}$. In addition, Proposition \ref{prop:max-no-s-disj} shows that we may use the estimate $\log M \le n\binom{sk}{k}$. Altogether we have
\begin{align}\label{max-no-s-disj}
\notag 2\log M+\max(N_1,N_2)-N_0 &\le 2n\binom{sk}{k}-\frac{1}{s+1}\binom{n-k-s+1}{k-1}\\
\notag &= 2n\binom{sk}{k}-\frac{k}{(s+1)(n-k-s+2)}\binom{n-k-s+2}{k}\\
&\le 2n\binom{sk}{k}\left[1-\frac{k}{2(s+1)(n-k-s+2)n}\left(\frac{n-k-s+2}{sk}\right)^k \right].
\end{align}
As $n\ge 2sk+38s^4$ and $s\ge 2$, we find $\frac{n-k-s+2}{sk} \ge \frac32 \left(1+\frac{74s^3}{3k}\right)$, and hence
\[
\left(\frac{n-k-s+2}{sk}\right)^{k-2}\ge \left(\frac32\right)^{k-2}\left(1+\frac{74s^3}{3k}\right)^{k-2}\ge \frac{k}{2}\cdot\frac{74(k-2)s^3}{3k}=\tfrac{37}{3}(k-2)s^3.
\]
This implies
\begin{align*}
\frac{k}{2(s+1)(n-k-s+2)n}\left(\frac{n-k-s+2}{sk}\right)^k&=\frac{n-k-s+2}{2(s+1)s^2kn}\left(\frac{n-k-s+2}{sk}\right)^{k-2}\\
&\ge \frac{37s(k-2)(n-k-s+2)}{6(s+1)kn} \ge \frac{37}{36}
\end{align*}
as $s/(s+1)\ge 2/3, (k-2)/k\ge 1/3$ and $(n-k-s+2)/n\ge 3/4$. Substituting this inequality into \eqref{max-no-s-disj}, we obtain
\[
2\log M + \max(N_1,N_2)-N_0 \le -\frac{1}{18}n\binom{sk}{k}\rightarrow -\infty.\qedhere
\]
\end{proof}

\subsection{Intersecting set systems}
In this section we shall use the removal lemma for disjoint sets (Lemma \ref{lem:setremoval}) to show that intersecting set systems in $\binom{[n]}{k}$ are typically trivial when $n \ge 2k + C \sqrt{k \ln k}$ for some positive constant $C$. Since the number of trivial intersecting families is 
\[
n\cdot 2^{{n-1\choose k-1}}\pm {n\choose 2} \cdot  2^{{n-2\choose k-2}}=(n+o(1)) 2^{{n-1\choose k-1}},
\]
it suffices to prove that there are $o(2^{\bekr})$ non-trivial intersecting families. 

We need a few classic theorems from extremal set theory. The first is a theorem of Hilton and Milner~\cite{hm67}, bounding the cardinality of a non-trivial uniform intersecting family. 

\begin{theorem}[Hilton and Milner]\label{thm:hilton-milner}
Let $\mc{F}\subset\binom{[n]}{k}$ be a non-trivial intersecting family with $k\ge 2$ and $n\ge 2k+1$. Then $\card{\mc{F}}\le \bekr-\bnkk+1$.
\end{theorem}

The next result we require is a theorem of Kruskal~\cite{Kruskal} and Katona~\cite{Katona}. For a family $\mc{F}\subset \binom{[n]}{r}$, its {\em $s$-shadow} in $\binom{[n]}{s}$, denoted $\partial^{(s)}\mc{F}$, is the family of those $s$-sets contained in some member of $\mc{F}$. For $x\in\bR$ and $r\in \bN$, we define the \emph{generalised binomial coefficient} $\binom{x}{r}$ by setting
\[
\binom{x}{r}=\frac{x(x-1)\hdots(x-r+1)}{r!}.
\]
The following convenient formulation of the Kruskal-Katona theorem is due to Lov\'asz~\cite{Lovasz-KK}.

\begin{theorem}[Lov\'asz]\label{thm:kruskal-katona}
Let $n, r$ and $s$ be positive integers with $s\le r\le n$. If $\mc{F}$ is a subfamily of $\binom{[n]}{r}$ with $\card{\mc{F}}=\binom{x}{r}$ for some real number $x\ge r$, then $\card{\partial^{(s)}\mc{F}}\ge \binom{x}{s}$.
\end{theorem}

With these results in hand, we now prove Theorem~\ref{thm:aa}.

\begin{proof}[Proof of Theorem \ref{thm:aa}]
The statement has been established for $n \ge 3k + 8 \ln k$ in \cite[Theorem 1.4]{BDDLS}, and so we may assume $n = 2k + s$ for some integer $s$ with $C \sqrt{k \ln k} \le s \le k + 8 \ln k$.

For each $\l \in \bN$, let $N_{\l}$ denote the number of maximal non-trivial intersecting families of size $\bekr - \l$.  By Theorem \ref{thm:hilton-milner}, we know $N_{\l} = 0$ for $\l < \bnkk - 1$.  By taking a simple union bound over the subfamilies of these families, we can bound the number of non-trivial intersecting families by
\[ \sum_{\l = \bnkk - 1}^{\bekr} N_{\l} 2^{\bekr - \l} = \left( \sum_{\l} N_{\l} 2^{-\l} \right) 2^{\bekr}, \]
so it suffices to show $\sum_{\l} N_{\l} 2^{-\l} = o(1)$.

By a result of Balogh et al. \cite[Proposition 2.2]{BDDLS}, we know the total number of maximal intersecting families can be bounded by $\sum_{\l} N_{\l} \le 2^{\frac 12 n \binom{2k}{k}}$, and so we have
\[ \sum_{\l \ge n \binom{2k}{k}} N_{\l} 2^{- \l} \le 2^{- n \binom{2k}{k}} \cdot \sum_{\l \ge n \binom{2k}{k} } N_{\l} \le 2^{- \frac12 n \binom{2k}{k}} = o(1). \]
Hence it suffices to show
\begin{equation} \label{eqn:bound}
\sum_{\l = \bnkk-1}^{n \binom{2k}{k}} N_{\l} 2^{- \l} = o(1).
\end{equation}

\medskip

We fix some integer $\l$ with $\bnkk - 1 \le \l \le n \binom{2k}{k}$, and fix some maximal intersecting family $\mc{F}$ of size $\bekr - \l$.  Let $\mc{S}$ be the star that minimises $\card{\mc{F}\De \mc{S}}$, and without loss of generality assume that $n$ is the center of $\mc{S}$. Let $\mc{A} = \mc{F} \setminus \mc{S}$, and $t = \card{\mc{A}}$.  Let $\mc{B} = \mc{S} \setminus \mc{F}$, and note that $\card{\mc{B}} = t + \l$.

Let $\mc{P} = \{ [n-1] \setminus A : A \in \mc{A} \}$, and observe that $\mc{P} \subseteq \binom{[n-1]}{n-k-1}$, since $n \notin A$ for all $A \in \mc{A}$.  Let $\mc{Q} = \{ B \setminus \{n\} : B \in \mc{B} \} \subseteq \binom{[n-1]}{k-1}$.  We claim that $\partial^{(k-1)} \mc{P} = \mc{Q}$.

Indeed, suppose $H \in \partial^{(k-1)} \mc{P}$.  Then there is some $A \in \mc{A}$ such that $H \subset [n-1] \setminus A$, and so $H \cap A = \emptyset$. As $n \notin A$, this forces $(\{n\} \cup H) \cap A = \emptyset$, and so $\{n\} \cup H \notin \mc{F}$. Hence $\{n\} \cup H \in \mc{B}$, giving $H \in \mc{Q}$.

For the opposite direction, suppose $H \notin \partial^{(k-1)} \mc{P}$.  Then, following the same argument as above, $(\{n\} \cup H) \cap A \neq \emptyset$ for all $A \in \mc{A}$.  By maximality of $\mc{F}$, we must have $\{n\} \cup H \in \mc{F}$, and thus $\{n\} \cup H \notin \mc{B}$, resulting in $H \notin \mc{Q}$.

We shall show that $\l \ge 2 n t$.  First let us see why this implies \eqref{eqn:bound}.  For each family $\mc{F}$ counted by $N_{\l}$, it suffices to provide the star $\mc{S}$ and the family $\mc{A}$ outside the star\footnote{For every choice of $\mc{F}$ there is a unique $\mc{A}$, but not every $\mc{A}$ corresponds to a maximal family $\mc{F}$}.  Indeed, since $\mc{Q} = \partial^{(k-1)} \mc{P}$, we can compute $\mc{F} \cap \mc{S}$, and hence completely determine $\mc{F}$.  Moreover, $\card{\mc{A}} = t \le \l / (2 n)$.  Thus
\[ N_{\l} 2^{-\l} \le n \cdot \binom{n-1}{k}^{\l / (2 n)} 2^{-\l} < n \cdot 2^{ - \l/2}, \textrm{ and so } \sum_{\l = \bnkk - 1}^{n \binom{2k}{k}} N_{\l} 2^{-\l} \le \frac{2n}{\sqrt{2} - 1} \cdot 2^{- \frac12 \bnkk } = o(1). \]

\bigskip

It remains to show $\l \ge 2nt$.  Letting $\mc{P}$ and $\mc{Q}$ be as above, recall that $\mc{Q} = \partial^{(k-1)} \mc{P}$. According to Theorem \ref{thm:kruskal-katona}, if $x$ is a real number so that $t = \card{\mc{P}} = \binom{x}{n-k-1}$, then $\l + t = \card{\mc{Q}} \ge \binom{x}{k-1}$.

Now observe that by Lemma~\ref{lem:setremoval}, we have $t \le C' n \l$ for some absolute constant $C'$.  Since $\l \le n \binom{2k}{k}$, this implies $t \le C' n^2 \binom{2k}{k}$.  Since $n = 2k+s$, we have $t = \binom{x}{n-k-1} = \binom{x}{k+s - 1}$. We next show that $x < 2k + \floor{\frac34 s}$. If not, then
\begin{align*}
	t = \binom{x}{k+s - 1} &\ge \binom{2k + \floor{\frac34 s}}{k + s - 1}= \binom{2k}{k}\cdot\frac{\binom{2k+s-1}{k+s-1}}{\binom{2k}{k}} \cdot \frac{\binom{2k + \floor{\frac34 s}}{k + s - 1}}{\binom{2k+s-1}{k+s-1}} \\
	&= \binom{2k}{k}\cdot\prod_{j=1}^{s-1} \frac{2k + j}{k + j} \cdot\prod_{j=\floor{\frac34 s}+1}^{s-1} \frac{k - s + 1+j}{2k + j} \\
	&\ge \binom{2k}{k}\left( \frac{2k + s}{k + s} \right)^{s-1} \left( \frac{ k - \frac14 s}{2k + \frac34 s} \right)^{\frac 14 s} .
\end{align*}
The bases of the exponential factors are minimised when $s$ is as large as possible; substituting $s \le k + 8 \ln k < 1.1k$, we can lower bound the coefficient of $\binom{2k}{k}$ by
\[ \frac{2.1}{3.1} \left( \frac{3.1}{2.1} \right)^s \left( \frac{0.725}{2.825} \right)^{\frac14 s} > \frac23 1.05^s > n^3 \]
as $s > C\sqrt{k \ln k} \ge 100 \ln n$, contradicting our upper bound $t \le C'n^2\binom{2k}{k}$.

Suppose, then, that $x \le 2k + \floor{\frac34 s}-1$. Since $t = \binom{x}{k + s - 1}$ and $\l + t \ge \binom{x}{k-1}$, we have
\[ 
\frac{\l}{t} \ge \frac{\binom{x}{k-1}}{\binom{x}{k + s - 1}} - 1 = \prod_{j = k}^{k + s - 1} \frac{ j }{x + 1 - j } - 1.\]
This product is decreasing in $x$, so we can substitute our upper bound $x \le 2k + \floor{\frac34 s} - 1$ to find
\[ 
\frac{\l}{t} \ge \prod_{j=k}^{k + s - 1} \frac{j}{2k + \floor{\frac34 s} - j } - 1 = \prod_{j=\floor{\frac34 s}+1}^{s-1} \frac{ k + j}{k + \floor{\frac34 s} - j } - 1 \ge \left( 1 + \frac{3s}{4k} \right)^{\frac14 s - 1} - 1. \]
This is increasing in $s$, so plugging in the lower bound $s \ge C \sqrt{k \ln k}$, we have
\[ \frac{\l}{t} \ge \left( 1 + \frac34 C \sqrt{ \frac{ \ln k }{k} } \right)^{\frac{C}{4} \sqrt{k \ln k} - 1} - 1 > e^{\frac{3C}{64} \ln k} - 1 \ge 8k > 2n,\]
as required.  This completes the proof.
\end{proof}



\section{Concluding remarks}\label{sec:concluding}

We close by offering some final remarks and open problems related to the supersaturation problems discussed in this paper.

\subsection{Supersaturation for permutations}

Theorem~\ref{thm-supsat-perm} shows, for $k \le c n^{1/2}$ and $s$ (very) close to $k(n-1)!$, one minimises the number of disjoint pairs in a family of $s$ permutations by selecting them from pairwise-disjoint cosets.  This leaves large gaps between the ranges where we know the answer to the supersaturation problem, and it would be very interesting to determine the correct behaviour throughout.  For instance, which family of $1.5(n-1)!$ permutations minimises the number of disjoint pairs?

Note that the derangement graph is $d_n$-regular, and so we can apply Lemma~\ref{lem:small-large} to determine the optimal families for sizes close to $k(n-1)!$ when $k \ge n - cn^{1/2}$ by taking complements.  However, the complement of a union of pairwise disjoint cosets is again a union of pairwise disjoint cosets, and hence there may well be a nested sequence of optimal families for this problem.  One candidate would be the initial segments of the lexicographic order on $S_n$, where $\pi < \si$ if and only if $\pi_j < \si_j$ for $j = \min \{ i \in [n] : \pi_i \neq \si_i \}$.

\subsection{Set systems of very large uniformity}  For set families, we improved the range of uniformities for which the small initial segments of the lexicographic order are known to be optimal.  In Corollary~\ref{cor:setclosetostar}, which applies when $n = \Omega(r^3 k^2)$, we handled the case where the family is a little larger than the union of $r$ stars.  However, if one restricts the size of the set families even further, one can obtain optimal bounds on $n$.  For instance, Katona, Katona and Katona~\cite{KKK} showed that adding one set to a full star is always optimal.

\begin{prop}\label{prop:katona}
Suppose $n \ge 2k+1$. Any system $\mc{F}\subseteq\binom{[n]}{k}$ with $\card{\mc{F}}=\bekr+1$ contains at least $\bnkk$ disjoint pairs.
\end{prop}

By applying the removal lemma (Lemma~\ref{lem:setremoval}), we can extend this exact result to a larger range of family sizes.

\begin{prop}
For some positive constant $c$, the following holds. Provided $n\ge 2k+2$ and $0 \le s \le \bekr+c\cdot\frac{n-2k}{n}\bnkk$, $\mathcal{L}(n,k,s)$ minimises the number of disjoint pairs among all systems of $s$ sets in $\binom{[n]}{k}$.
\end{prop}
\begin{proof}
Let $C$ be the positive constant from Lemma~\ref{lem:setremoval} and set $c=(20C)^{-2}$. Suppose $\mc{F}\subseteq \binom{[n]}{k}$ is a family with $\card{\mc{F}}=\bekr+t$ for some $1\le t \le c\cdot\frac{n-2k}{n}\bnkk$. Letting $s=\bekr+t$, we shall show that $\disj(\mc{F})\ge \disj(\mc{L}_{n,k}(s))=t \bnkk$. Suppose otherwise that $\disj(\mc{F})<t \bnkk$. By Lemma~\ref{lem:setremoval}, there exists a star $\mc{S}$ such that $\card{\mc{F}\De\mc{S}}\le \frac{1}{2}\bnkk$. 
It follows that $\card{\mc{F}\cap\mc{S}}=\bekr-p$ for some integer $p$ with $0\le p\le \frac12\bnkk$. As $\card{\mc{F}}=\bekr+t$ and $\card{\mc{F}\cap\mc{S}}=\bekr-p$, we must have $\card{\mc{F}\setminus\mc{S}}=p+t$. Since each set in $\mc{F}\setminus\mc{S}$ is disjoint from exactly $\bnkk$ sets in the star $\mc{S}$ and $\card{\mc{F}\cap\mc{S}}=\bekr-p$, we conclude $\disj(F,\mc{F}\cap \mc{S}) \ge \bnkk-p>0$ for all $F\in \mc{F}\setminus\mc{S}$. Thus
\begin{align*}
\disj(\mc{F})&\ge \sum_{F\in \mc{F}\setminus \mc{S}}\disj(F,\mc{F}\cap \mc{S})\ge \card{\mc{F}\setminus\mc{S}}\left(\bnkk-p\right)\\
&=(p+t)\left(\bnkk-p\right)
=t\bnkk+p\left(\bnkk-p-t\right)\\
&\ge t\bnkk,
\end{align*}
where the last inequality holds since $p\le \frac12\bnkk$ and $t \le c\cdot\frac{n-2k}{n}\bnkk$.
\end{proof}

\subsection{A counterexample to the Bollob\'as--Leader conjecture}

Finally, it remains to extend the set supersaturation results to larger values of $k$.  Are small initial segments of the lexicographic order still optimal when $k>\sqrt{n}$?

This is not the case when $n=3k-1$, as the following construction shows. Let $s = \binom{n-1}{k-1} + \binom{2k-1}{k} - 1$.  Then $\mc L(n,k,s)$ consists of one full star, and $\binom{2k-1}{k} - 1$ sets from another star, each of which is disjoint from $\binom{n-k-1}{k-1} = \binom{2k-2}{k-1}$ sets from the full star.  Hence $\disj(\mc L(n,k,s)) = \left( \binom{2k-1}{k} - 1 \right) \binom{2k-2}{k-1}$.

Now instead let $\mc F'$ be the family consisting of the $\mc S_1$, the full star with centre $1$, and all but one $k$-element subset of $\{2, 3, \hdots, 2k \}$.  Since $\mc F'$ again consists of a full star and an intersecting family of size $\binom{2k-1}{k} - 1$, we have $\disj(\mc F') = \disj(\mc L(n,k,s))$.  Now form the family $\mc F$ from $\mc F'$ by replacing the set $A = \{1, 2k+1, \hdots, 3k-1\}$ with the missing $k$-set $B$ from $\{2, 3, \hdots, 2k\}$.  We lose $\binom{2k-1}{k}-1$ disjoint pairs when we remove $A$, and gain only $\binom{n-k-1}{k-1} - 1 = \binom{2k-2}{k-1} - 1$ disjoint pairs when we add $B$.  As $\binom{2k-2}{k-1} < \binom{2k-1}{k}$, it follows that $\disj(\mc F) < \disj(\mc L(n,k,s))$, showing the initial segment of the lexicographic order is not optimal.

Bollob\'as and Leader~\cite{Bol-L} conjectured that the solution to the supersaturation problem is always given by an $\ell$-ball.  Given $n, k$ and $s$, an $\ell$-ball of size $s$ is a family $\mc B_{\ell}(n,k,s)$ of $s$ sets such that there is some $r$ with $\left\{ F \in \binom{[n]}{k} : \card{F \cap [r]} \ge \ell \right\} \subseteq \mc B_{\ell}(n,k,s) \subseteq \left\{ F \in \binom{[n]}{k} : \card{F \cap [r+1]} \ge \ell \right\}$.  In particular, the initial segments of the lexicographic order are $1$-balls, while their complements are isomorphic to $k$-balls.

We have shown that the construction $\mc F$ given above has fewer disjoint pairs than the $1$-balls of size $s = \card{\mc F}$.  Computer-aided calculations show that for $n = 3k-1$, $s = \binom{n-1}{k-1} + \binom{2k-1}{k} - 1$ and $5 \le k \le 15$, the $1$-balls have far fewer disjoint pairs than the $\ell$-balls for $\ell \ge 2$, showing that $\mc F$ gives a counterexample to the Bollob\'as--Leader conjecture for these parameters.  The numerical evidence suggests that $\mc F$ should be a counterexample for all $k \ge 5$, but it is difficult to estimate the number of disjoint pairs in $\mc B_{\ell}(3k-1,k,s)$ for $\ell \ge 2$, and so we have been unable to prove this.

\subsection*{Acknowledgement} We would like to thank the anonymous referee for their several valuable suggestions for improving the presentation of this paper.


\medskip

{\footnotesize \obeylines \parindent=0pt
	
J\'ozsef Balogh, Department of Mathematical Sciences, University of Illinois at Urbana-Champaign, IL, USA, and Moscow Institute of Physics and Technology, 9 Institutskiy per., Dolgoprodny, Moscow Region, 141701, Russian Federation.

Shagnik Das, Institut f\"ur Mathematik, Freie Universit\"at Berlin, Germany.

Hong Liu and Maryam Sharifzadeh, Mathematics Institute, University of Warwick, UK.

Tuan Tran,  Department of Mathematics, ETH, Switzerland.

}

\begin{flushleft}
	{\it{E-mail addresses}:
		\tt{jobal@illinois.edu, shagnik@mi.fu-berlin.de, $\lbrace$h.liu.9,~m.sharifzadeh$\rbrace$@warwick.ac.uk, manh.tran@math.ethz.ch}}
\end{flushleft}
%

\appendix

\section{Intersection graphs} \label{app:int-graphs}

In this appendix, we prove Proposition~\ref{prop:int-graphs}, which shows how the intersection graph determines various parameters about the corresponding union of cosets, including its size and number of disjoint pairs.

\subsection{Some preliminaries}

We start by introducing some further notation we will use throughout this appendix.  First, recall that $d_n$ denotes the number of derangements in $S_n$, and that $D_n = d_n + d_{n-1}$.  It will also be convenient for us to define the parameter $D'_n = d_n + 2d_{n-1}$, a quantity that arise later in our proof.

Next, given a graph $G$, $k_t(G)$ denotes the number of $t$-cliques in $G$.  We will further write $K_t(G)$ for the set of these $t$-cliques.  Moreover, given a vertex subset $X \subseteq V(G)$, we denote by $k_{t,X}(G)$ the number of $t$-cliques in $G$ that contain $X$.  In particular, we have $k_t(G) = k_{t,\emptyset}(G) = \card{K_t(G)}$.  Again, we will omit $G$ from the notation when the graph is clear from the context.

Finally, $\bar{P_3}$ is the complement of the path on three vertices, which is the union of an edge and an isolated vertex.  Let $\1_{\bar{P_3}}: V(G) \times \binom{V(G)}{2} \rightarrow \{0,1\}$ be the function defined by setting $\1_{\bar{P_3}}(x,\{y,z\}) = 1$ if and only if $yz$ is the only edge of the induced subgraph $G[\{x,y,z\}]$.  We then denote the number of induced copies of $\bar{P_3}$ in $G$ by $i(\bar{P_3}, G)$, noting that $i(\bar{P_3}, G) = \sum_{x \in V(G)} \sum_{\{y,z\} \in \binom{V(G)}{2}} \1_{\bar{P_3}}(x,\{y,z\})$.

With this additional notation in place, we close these preliminaries with the following crucial observation, which we shall make repeated use of.
\begin{obs}\label{obs:perm}
	If $G$ is the intersection graph of a union of cosets, the following properties hold.
	\begin{itemize}
		\item[\rm (i)] For every subset $X\subset V(G)$, the intersection $\cap_{x\in X}\mc{T}_x$ is non-empty if and only if $G[X]$ is a clique. In this case, $\card{\cap_{x\in X}\mc{T}_x}=(n-\card{X})!$.
		\item[\rm (ii)] If $n \ge 10 \ell^2$, $(i_1,j_1),\ldots,(i_{\l},j_{\l})$ form an $\l$-clique in $G$, and $\pi \in S_n \setminus \bigcup_{s=1}^{\l} \mc T_{(i_s,j_s)}$, then 
		\begin{align*} 
			\disj \left( \pi,\bigcap\nolimits_{s=1}^{\l} \mc{T}_{(i_s,j_s)} \right) =& d_{n-\l}+\left(\l-\card{\{i_1,\ldots,i_{\l}\}\cap \{\pi^{-1}(j_1),\ldots,\pi^{-1}(j_{\l})\}}\right)d_{n-\l-1}\\
			&\pm 7 \l^2(n-\l-2)!.
		\end{align*}
	\end{itemize}
\end{obs}

\begin{proof}
	(i) Since $\mc{T}_{(i,j)}\cap \mc{T}_{(i',j')}=\emptyset$ whenever $(i,j)$ and $(i',j')$ are not adjacent in $G$, we have $\cap_{x\in X}\mc{T}_x=\emptyset$ whenever $G[X]$ is not a clique. Now suppose that $X=\{(i_1,j_1),\ldots,(i_{\l},j_{\l})\}$ spans a clique in $G$. Then we must have $\card{\{i_1,\ldots,i_{\l}\}}=\card{\{j_1,\ldots,j_{\l}\}}=\l$. Hence $\card{\cap_{x\in X}\mc{T}_x}$ is the number of bijections from $[n]\setminus \{i_1,\ldots,i_{\l}\}$ to $[n]\setminus \{j_1,\ldots,j_{\l}\}$, which is $(n-\l)!$.
	
	(ii) Fix a permutation $\pi\in S_n\setminus \left(\mc{T}_{(i_1,j_1)}\cup \ldots \cup \mc{T}_{(i_{\l},j_{\l})}\right)$, that is, a permutation satisfying $\pi(i_s)\neq j_s$ for all $s\in[\l]$. If a permutation $\sigma \in \cap_{s = 1}^{\ell} \mc T_{(i_s, j_s)}$ intersects $\pi$, then we must have $\sigma(x) = \pi(x)$ for some $x \in [n]$.  Let $\mc A_x$ denote the family of such permutations, and observe $\mc A_x = \mc T_{(x,\pi(x))} \cap \left( \cap_{s=1}^{\ell} \mc T_{(i_s,j_s)} \right)$.  
	By part (i), for this intersection to be non-empty, we require $(x,\pi(x))$ to be adjacent to each $(i_s, j_s)$ in the intersection graph, or, equivalently, $x \notin \{ i_1, \hdots, i_{\ell} \} \cup \{ \pi^{-1}(j_1), \hdots, \pi^{-1}(j_{\ell}) \}$.  For brevity, set
	\begin{align*}
		c &= \l-\card{\{i_1,\ldots,i_{\l}\}\cap \{\pi^{-1}(j_1),\ldots,\pi^{-1}(j_{\l})\}} \text{ and } \\
		I &= [n]\setminus\left(\{i_1,\ldots,i_{\l}\}\cup\{\pi^{-1}(j_1),\ldots,\pi^{-1}(j_{\l})\}\right).
	\end{align*}
	Note that $\card{I} = n - \ell - c$, and $\cup_{x \in I} \mc A_x$ is the family of all permutations $\sigma \in \cap_{s = 1}^{\ell} \mc T_{(i_s, j_s)}$ that intersect $\pi$. Applying the inclusion-exclusion formula and part (i), we obtain
	\begin{equation}\label{eqn:obs-ii}
	\disj\left(\pi,\cap_{s=1}^{\l} \mc{T}_{(i_s,j_s)}\right)=\card{\cap_{s=1}^{\l} \mc{T}_{(i_s,j_s)}\setminus\cup_{x\in I}\mc{A}_x}=\sum_{i=0}^{n-\l-c}(-1)^i\binom{n-\l-c}{i}(n-\l-i)!.
	\end{equation}
	
	Let $\al_i$ be the real number such that $\binom{n-\l-c}{i}(n-\l-i)!=\al_i \frac{(n-\l)!}{i!}$. We shall approximate $\al_i$ by a simple function, and then use it to compute $\disj\left(\pi,\cap_{s=1}^{\l} \mc{T}_{(i_s,j_s)}\right)$. By definition,
	\begin{align*}
	\al_i=\frac{(n-\l-c)!(n-\l-i)!}{(n-\l)!(n-\l-c-i)!}=\frac{(n-\l-i)\cdots(n-\l-i-c+1)}{(n-\l)\cdots(n-\l-c+1)}=\prod_{j=0}^{c-1}\left(1-\frac{i}{n-\l-j}\right).
	\end{align*}
	Hence $\al_i\ge \left(1-\frac{i}{n-\l-c+1}\right)^c \ge 1-\frac{ci}{n-\l-c+1}= 1-\frac{ci}{n-\l}-\frac{c(c-1)i}{(n-\l)(n-\l-c+1)}$. On the other hand, $\al_i \le \left(1-\frac{i}{n-\l}\right)^c \le \exp\left\{-\frac{ci}{n-\l}\right\} \le 1-\frac{ci}{n-\l}+\left(\frac{ci}{n-\l}\right)^2$. Since $c \le \ell$ and $n \ge 10\l^2$, we may write $\al_i=1-\frac{ci}{n-\l}+\ep_i$ with $|\ep_i|\le\frac{c^2i^2}{(n-\l)^2}$, providing an effective estimate when $i$ is small.
	
	Plugging these estimates into~\eqref{eqn:obs-ii}, we have
	\begin{align*}
		\disj \left( \pi, \cap_{s=1}^{\ell} \mc T_{(i_s,j_s)} \right) &= \sum_{i=0}^{n-\ell - c} (-1)^i \alpha_i \frac{(n-\ell)!}{i!} \\
		&= \sum_{i=0}^{n - \ell - c} (-1)^i \left( 1 - \frac{ci}{n - \ell} \right) \frac{(n-\ell)!}{i!} \pm \sum_{i=0}^{n - \ell - c} \card{\ep_i} \frac{(n- \ell)!}{i!} \\
		&= \sum_{i=0}^{n-\ell} (-1)^i \left( 1 - \frac{ci}{n- \ell} \right) \frac{(n-\ell)!}{i!} \pm \sum_{i=n-\ell - c + 1}^{n- \ell} \card{ 1 - \frac{ci}{n - \ell} } \frac{(n-\ell)!}{i!} \\
		& \quad \pm \sum_{i=0}^{n-\ell - c} \frac{c^2 i^2 (n-\ell)!}{(n-\ell)^2 i!}.
	\end{align*}
	
	Expanding the leading term gives
	\[ \sum_{i=0}^{n - \ell} (-1)^i \frac{(n-\ell)!}{i!} + c \sum_{i=0}^{n-\ell-1} (-1)^i \frac{(n-\ell-1)!}{i!} = d_{n-\ell} + c \cdot d_{n-\ell-1}.\]
	To bound the first error term, observe that if $n - \ell - c + 1 \le i \le n - \ell$ (for which we must have $c \ge 1$), we have $\card{1-\frac{ci}{n-\l}}\le c$ and $\frac{(n-\l)!}{i!}\le n^{c-1}$, and so
	\[
	\sum_{i = n-\l-c+1}^{n-\l}\card{1-\frac{ci}{n-\l}} \frac{(n-\l)!}{i!} \le c^2 n^{c-1} \le \l^2 n^{c-1} \le \ell^2 (n-\ell - 2)!.
	\] 
	Finally, using $c \le \ell$, we bound the second error term by
	\[ \sum_{i=0}^{n-\l - c}\frac{c^2 i^2 (n-\l)!}{(n-\ell)^2 i!} \le c^2 (n-\l-2)!\sum_{i\ge 0}\frac{i^2}{i!} \le 6 \l^2(n-\l-2)!. \]
	Putting these bounds together gives the desired expression for $\disj \left( \pi, \cap_{s=1}^{\ell} \mc T_{(i_s, j_s)} \right)$.
\end{proof}

\subsection{Proof of Proposition~\ref{prop:int-graphs}(a)}

Armed with these preliminaries, we may begin to prove the statements in Proposition~\ref{prop:int-graphs}, of which the first is by far the simplest.

\begin{proof}[Proof of Proposition~\ref{prop:int-graphs}(a)]
By the Bonferroni inequalities, we have
\[ \card{\mc{G}} = \card{\bigcup_{x\in V(G)}\mc{T}_x} = \sum_{1\le i\le 3}(-1)^{i-1}\left(\sum_{X\in \binom{V(G)}{i}}\card{\bigcap_{x\in X}\mc{T}_x}\right) \pm \sum_{X\in \binom{V(G)}{4}}\card{\bigcap_{x\in X}\mc{T}_x}. \]
By Observation~\ref{obs:perm}(i), $\cap_{x \in X} \mc T_x$ is empty unless $X$ induces a clique in $G$, in which case $\card{\cap_{x \in X} \mc T_x} = (n - \card{X})!$.  Hence we have
\[ \card{\mc G} = k_1 (n-1)! - k_2 (n-2)! + k_3(n-3)! \pm k_4 (n-4)!, \]
as required.
\end{proof}

\subsection{Proof of Proposition~\ref{prop:int-graphs}(b)}

In the second part of the proposition, we count the number of disjoint pairs between $\mc G$ and an arbitrary permutation $\pi \in S_n \setminus \mc G$.  We will in fact prove the more accurate estimate given in the claim below, as this will be required in the proof of part (c).

\begin{claim} \label{clm:intgraphsb}
Let $\mc G$ be a union of $k_1 \le n$ cosets with intersection graph $G$.  Then, for every $\pi \in S_n \setminus \mc G$,
\begin{align*}
	\disj(\pi, \mc G) =& \; k_1 D_{n-1} - k_2 D'_{n-2} + \left( k_3 + \sum_{xy \in E(G)} \card{ \{x_1, y_1\} \cap \{ \pi^{-1}(x_2), \pi^{-1}(y_2) \} } \right) d_{n-3} \\
	&\pm (28 k_2 + 4 k_3 + k_4) (n-4)!,
\end{align*}
where the indices $x = (x_1, x_2)$ and $y = (y_1, y_2)$ are vertices of $G$.
\end{claim}

We first verify that this implies the bound from the proposition.

\begin{proof}[Proof of Proposition~\ref{prop:int-graphs}(b)]
The leading term is already in the desired form.  For the second-order term, observe that $D'_{n-2} = D_{n-2} + d_{n-3}$.  In the third term, we bound the coefficient of $d_{n-3}$ above by $k_3 + 2k_2$, and recall that we also have a term of $-k_2 d_{n-3}$ from the second-order term.  Thus, in total, the third-order term is at most $(k_3 + k_2)d_{n-3}$.  Since $k_3 \le k_1 k_2$ and $d_{n-3} \le (n-3)!$, we can bound this from above by $2 k_1 k_2 (n-3)!$.

With regards to the error term, note that by making the constant $c$ in Proposition~\ref{prop:int-graphs} sufficiently small, we may assume $n$ is large.  Thus, using the aforementioned bound on $k_3$, and bounding $k_4 \le k_1^2 k_2 / 12$, we have $(28 k_2 + 4k_3 + k_4)(n-4)! \le k_1 k_2 (n-3)!$.

Substituting these estimates into the equation from Claim~\ref{clm:intgraphsb} gives $\disj(\pi, \mc G) = k_1 D_{n-1} - k_2 D_{n-2} \pm 3k_1 k_2 (n-3)!$, as required.
\end{proof}

We now prove the claim.

\begin{proof}[Proof of Claim~\ref{clm:intgraphsb}]
Let $\pi$ be an arbitrary permutation in $S_n\setminus \mc{G}$. It follows from the Bonferroni inequalities and Observation~\ref{obs:perm}(i) that 
	\begin{equation*}
	\disj(\pi,\mc{G}) = \sum_{i=1}^{3}(-1)^{i-1}\left(\sum_{\substack{X \in K_i(G)}}\disj \left(\pi,\cap_{x\in X}\mc{T}_x \right)\right)\pm \sum_{\substack{X \in K_4(G)}}\disj\left(\pi,\cap_{x\in X}\mc{T}_x\right).
	\end{equation*}
	
	Recall from Subsection~\ref{subsec:derangement} that $\disj(\pi,\mc{T}_x)=D_{n-1}$ for every $x\in V(G)$. For $xy\in E(G)$, say $x=(x_1,x_2)$ and $y=(y_1,y_2)$, Observation~\ref{obs:perm}(ii) implies 
\begin{align*}
\disj(\pi,\mc{T}_x\cap T_y) &= d_{n-2} + \left( 2 - \card{\{x_1,y_1\}\cap\{\pi^{-1}(x_2),\pi^{-1}(y_2)\}} \right) d_{n-3} \pm 28 (n-4)!  \\
	&=D_{n-2}'-\card{\{x_1,y_1\}\cap\{\pi^{-1}(x_2),\pi^{-1}(y_2)\}} d_{n-3}\pm 28(n-4)!
\end{align*}
as $D_{n-2}'=d_{n-2}+2d_{n-3}$. 
Again using Observation~\ref{obs:perm}(ii) gives $\disj\left(\pi,\cap_{x\in X}\mc{T}_x\right)=d_{n-3}\pm 4(n-4)!$ for every $X\in K_3(G)$. For each $X\in K_4(G)$, we may deduce from Observation \ref{obs:perm} (i) that
	$\disj(\pi,\cap_{x\in X}\mc{T}_x)\le \card{\cap_{x\in X}\mc{T}_x}=(n-4)!$. Combining these bounds gives the claim.
\end{proof}

\subsection{Proof of Proposition~\ref{prop:int-graphs}(c)}

In the third part of the proposition, we count the number of disjoint pairs within a union $\mc G$ of cosets, with the result depending on numerous parameters of the intersection graph $G$.  We shall once more prove a more precise estimate that we will need in the proof of part (d).

\begin{claim}\label{clm:intgraphsc}
Let $\mc G$ be a union of $k_1 \le c n^{1/2}$ cosets with intersection graph $G$.  Then
\[ \disj(\mc G) = \sum_{i=1}^3 a_i (n-i)! D_{n-1} \pm 12 k_1^3 k_2 (n-1)! (n-4)!, \]
where
\begin{align*}
a_1 &= \binom{k_1}{2}, \\
a_2 &= -(k_1 - 1)k_2, \textrm{and} \\
a_3 &= \frac12 \left( (2k_1 - 3)k_3 + (k_2 - k_1 + 1)k_2 + i(\bar{P_3},G) \right).
\end{align*}
\end{claim}

Let us first verify that this claim suffices for the proposition.

\begin{proof}[Proof of Proposition~\ref{prop:int-graphs}(c)]
The first two terms in Claim~\ref{clm:intgraphsc} are exactly as required.  We need only verify that the sum of the other terms is at most $2 k_1^2 k_2 (n-1)! (n-3)!$ in magnitude.  This is easily seen to be true, using the bounds $(2k_1 - 3)k_3 \le \frac23 k_1^2 k_2$, $k_2^2 \le \frac12 k_1^2 k_2$, and $i(\bar{P_3},G) \le k_1 k_2 \le \frac12 k_1^2 k_2$, and recalling that $k_1 \le cn^{1/2}$ for some small constant $c$.
\end{proof}

We now prove the claim.

\begin{proof}[Proof of Claim~\ref{clm:intgraphsc}]

The idea behind the proof is to partition the permutations in $\mc G$ based on how many of the cosets they are contained in.  For each vertex set $X\subseteq V(G)$, let $\mc{M}_X$ be the family of all permutations $\pi \in \mc{G}$ which satisfy $\{x\in V(G):\pi \in \mc{T}_x\}=X$.

We shall use the symbol $\dot\cup$ to denote an union of disjoint sets. From Observation \ref{obs:perm} (i) we find $\mc{M}_i:=\dot\cup_{\card{X}=i}\mc{M}_X=\dot\cup_{X\in K_i(G)}\mc{M}_X$, resulting in $\disj(\mc{M}_i,\mc{G})= \sum_{X\in K_i(G)}\disj(\mc{M}_X,\mc{G})$. The following claim evaluates these expressions.
	
	\begin{claim}\label{clm:Midisjoint}
		For $\mc G$ and $\mc M_i$ as defined above,
		\begin{itemize}
			\item[(i)] $\sum_{i\ge 4}\disj(\mc{M}_i,\mc{G}) \le k_1k_4(n-1)!(n-4)!$,
			\item[(ii)] $\disj(\mc{M}_3,\mc{G})=(k_1-3)k_3 (n-3)!D_{n-1}\pm 3k_{1}^2k_3 (n-1)!(n-4)!$,
			\item[(iii)] $\disj(\mc{M}_2,\mc{G})=\sum\limits_{i=2}^{3}c_i (n-i)!D_{n-1}\pm 16k_1k_2^2 (n-1)!(n-4)!$, where $c_2=(k_1-2)k_2$ and $c_3=-3(k_1-2)k_3-\sum\limits_{X\in K_2(G)}k_2(G-X)$, and
			\item[(iv)] $\disj(\mc{M}_1,\mc{G})=\sum\limits_{i=1}^{3}b_i (n-i)!D_{n-1}\pm  13 k_1^3 k_2 (n-1)!(n-4)!$, in which ${b_1=k_1(k_1-1)}$, $b_2=(-3k_1+4)k_2$, and $b_3=(4k_1-6)k_3-(k_1-2)k_2+i({\bar{P_3}},G)+\sum\limits_{x\in V(G)}k_{2,\{x\}}k_2(G-x)$.
		\end{itemize}
	\end{claim}
	
	Since $\mc{G}=\dot\cup_{i\ge 1}\mc{M}_i$, one has $\disj(\mc{G}) = \tfrac12 \sum_{1\le i \le k_1}\disj(\mc{M}_i,\mc{G})$.  Claim~\ref{clm:intgraphsc} thus follows by summing the above, noting that $\sum\limits_{x}k_{2,\{x\}} k_2(G-x)-\sum\limits_{X\in K_2(G)}k_2(G-X)=k_2(k_2-1)$.
\end{proof}

It remains to prove Claim~\ref{clm:Midisjoint}.  We begin by bounding the contribution from permutations in at least four cosets.

\begin{proof}[Proof of Claim~\ref{clm:Midisjoint}(i)]
	One has $\card{\mc{G}}\le k_1(n-1)!$, and $\card{\cup_{i\ge 4}\mc{M}_i}\le \sum_{X\in K_4(G)}\card{\cap_{x\in X}\mc{T}_x}\le k_4(n-4)!$ due to Observation \ref{obs:perm} (i). Hence $\sum_{i\ge 4}\disj(\mc{M}_i,\mc{G})\le \card{\cup_{i\ge 4}\mc{M}_i}\card{\mc{G}} \le k_1k_4 (n-1)!(n-4)!$.
\end{proof}

We next consider permutations in exactly three cosets.

\begin{proof}[Proof of Claim~\ref{clm:Midisjoint}(ii)]
	Fix $X\in K_3(G)$ and let $\pi \in \mc{M}_X$.  Observe that if $\sigma \in \mc T_x$ for $x \in X$, then $\pi$ and $\sigma$ intersect.  Thus, when counting the disjoint pairs between $\pi$ and $\mc G$, we need only consider the subfamily of $\mc G$ corresponding to the intersection graph given by $G - X$.  Applying Proposition~\ref{prop:int-graphs}(b) to this subgraph, we see that 
	\[ \disj(\pi,\mc{G})=\disj(\pi,\cup_{x\in V(G)-X}\mc{T}_x)=(k_1-3)D_{n-1}\pm 2k_2 (n-2)!. \] 
	On the other hand, it follows from the Bonferroni inequalities and Observation \ref{obs:perm} (i) that $\card{\mc{M}_X}=(n-3)!\pm k_{4,X} (n-4)!$ for every triangle $X\in K_3(G)$. Combining with the trivial bound $k_{4,X}\le k_1$, this gives 
	\[ \card{\mc{M}_3}=\sum_{X\in K_3(G)}\card{\mc{M}_X}=k_3 (n-3)!\pm k_1k_3 (n-4)!. \]
		Summing $\disj(\pi,\mc{G})$ over all permutations $\pi \in \mc{M}_3$ and using the estimate $2k_2\le k_1^2$ gives the desired result.
\end{proof}
	
We shall use a similar counting argument to estimate $\disj(\mc{M}_2,G)$. 

\begin{proof}[Proof of Claim~\ref{clm:Midisjoint}(iii)]
	Fix $X \in K_2(G)$ and $\pi \in \mc M_X$.  As before, any disjoint permutations come from cosets in $G - X$.  Applying Proposition~\ref{prop:int-graphs}(b) to $G-X$ gives
		\[
		\disj(\pi,\mc{G})=(k_1-2)D_{n-1}-k_2(G-X)D_{n-2} \pm 4k_1 k_2 (n-3)!.
		\]
		On the other hand, by appealing to the Bonferroni inequalities and Observation~\ref{obs:perm}(i), 	we see that $\card{\mc{M}_X}=(n-2)!-k_{3,X}\cdot (n-3)!\pm k_{4,X} (n-4)!$. These two bounds, together with the estimates $k_{3,X} \le k_1 < n$ and $k_{4,X} \le k_2 \le k_1^2$, imply
		\begin{align*}
		\disj(\mc{M}_X,\mc{G})=&(k_1-2)(n-2)!D_{n-1}-(k_1-2)k_{3,X} (n-3)!D_{n-1}\\
		&-k_2(G-X) (n-2)!D_{n-2}\pm 15 k_1 k_2 (n-1)!(n-4)!.
		\end{align*} 
As $(n-2)!D_{n-2}=(n-3)!D_{n-1}\pm (n-1)!(n-4)!$, this expression can be simplified further as
		\begin{align*}
		\disj(\mc{M}_X,\mc{G})=&(k_1-2)(n-2)!D_{n-1}-\left((k_1-2)k_{3,X}+k_2(G-X)\right) (n-3)!D_{n-1}\\
		& \pm 16k_1k_2 (n-1)!(n-4)!.
		\end{align*} 
		Summing $\disj(\mc{M}_X,\mc{G})$ over all $X\in K_2(G)$ and using the identity $\sum_{X \in K_2(G)}k_{3,X}=3k_3$ results in the desired equation.
\end{proof}

	Finally we come to what is, in some sense, the trickiest part of our proof, which is dealing with permutations in a single coset.
	
\begin{proof}[Proof of Claim~\ref{clm:Midisjoint}(iv)]
Fix a vertex $x\in V(G)$ and a permutation $\pi \in \mc M_{\{x\}}$. Once again, any disjoint permutations must come from cosets in $G - x$.  Applying Claim~\ref{clm:intgraphsb} to this subgraph, we find
		\begin{align*}
			\disj(\pi,\mc{G})=&(k_1-1)D_{n-1}-k_2(G-x)D'_{n-2}\\
			&+\left( k_3(G-x)+\sum_{(y_1,y_2)(z_1,z_2)\in E(G-x)}\card{\{y_1,z_1\}\cap \{\pi^{-1}(y_2),\pi^{-1}(z_2)\}}\right)d_{n-3}\\
			&\pm (28k_2+4k_3+k_4) (n-4)!.	
		\end{align*}
		
On the other hand, from the Bonferroni inequalities we have  
		\[
		\card{\mc{M}_{\{x\}}}= (n-1)! + \sum_{i \in \{2,3\}}(-1)^{i-1}k_{i,\{x\}} (n-i)!\pm k_{4,\{x\}} (n-4)!.
		\]
		
For each edge $(y_1,y_2)(z_1,z_2)\in E(G-x)$, we observe that for all but at most $2(n-2)!$ permutations $\pi \in \mc M_{\{x\}}$, we have $\card{\{y_1,z_1\}\cap \{\pi^{-1}(y_2),\pi^{-1}(z_2)\}}=\1_{\bar{P_3}}(x,\{y,z\})$.  Indeed, $\1_{\bar{P_3}}(x,\{y,z\}) = 1$ if and only if $x \in \{(y_1, z_2), (z_1, y_2)\}$, in which case, since $\pi(x_1) = x_2$, we have $\card{\{y_1,z_1\}\cap \{\pi^{-1}(y_2),\pi^{-1}(z_2)\}} \ge 1$, with equality unless both $\pi(y_1) = z_2$ and $\pi(z_1) = y_2$.  When $x \notin \{(y_1, z_2), (z_1, y_2)\}$, in order for $\card{\{y_1,z_1\}\cap \{\pi^{-1}(y_2),\pi^{-1}(z_2)\}}$ to be positive, we need $\pi(y_1) = z_2$ or $\pi(z_1) = y_2$ in addition to $\pi(x_1) = x_2$, thus giving at most $2(n-2)!$ exceptions in this case.
	
		Putting these facts together and summing $\disj(\pi,\mc{G})$ over all $\pi \in \mc{M}_{\{x\}}$ then gives
		\begin{align*}
			\disj(\mc{M}_{\{x\}},\mc{G}) =& (k_1-1) (n-1)!D_{n-1}- k_2(G-x) (n-1)!D'_{n-2}-(k_1-1)k_{2,\{x\}} (n-2)!D_{n-1}\\
			&+ k_{2,\{x\}} k_2(G-x) (n-2)!D'_{n-2} + (k_1-1)k_{3,\{x\}} (n-3)!D_{n-1} \\
			&+ \left( k_3(G-x) + \sum_{yz \in E(G-x)}\1_{\bar{P_3}}(x,\{y,z\}) \right) (n-1)! d_{n-3} \pm 4 k_2 (n-2)!d_{n-3}\\
			&\pm 10 k_1^2 k_2 (n-1)!(n-4)!,
		\end{align*}
		where in the final term we use $k_i \le 2 k_1^{i-2} k_2 / i!$ and the fact that $k_1 \le c n^{1/2}$ to bound the lower-order error terms.
		
		Moreover, we have the identities $(n-1)!D'_{n-2}=(n-2)!D_{n-1}+(n-3)!D_{n-1}\pm (n-1)!(n-4)!$, $(n-2)!D_{n-2}'=(n-3)!D_{n-1}\pm (n-1)!(n-4)!$, and $(n-1)!d_{n-3}=(n-3)!D_{n-1}\pm (n-1)!(n-4)!$. Hence 
		\[
		\disj(\mc{M}_{\{x\}},\mc{G})=\sum_{i=1}^{3}b_i' (n-i)!D_{n-1}\pm 13 k_1^2 k_2 (n-1)!(n-4)!,
		\]	
		where
		\begin{align*}
			b_1' &= k_1-1,\\
			b_2' &= -k_2(G-x)-(k_1-1)k_{2,\{x\}}, \textrm{and}\\
			b_3' &=-k_2(G-x)+k_{2,\{x\}} k_2(G-x) +(k_1-1)k_{3,\{x\}}+ k_3(G-x)+\sum_{yz\in E(G-x)}\1_{\bar{P_3}}(x,\{y,z\}).
		\end{align*}
		Noting that $\sum\limits_{x} k_2(G-x)=(k_1-2)k_2$, $\sum\limits_{x} k_{2,\{x\}}=2k_2$, $\sum\limits_{x} k_3(G-x)=(k_1-3)k_3$, and $\sum\limits_{x} k_{3,\{x\}}=3k_3$, and summing the above estimate for $\disj(\mc{M}_{\{x\}},\mc{G})$ over all $x\in V(G)$, we get the desired formula for $\disj(\mc{M}_1,\mc{G})$.
\end{proof}	
	
\subsection{Proof of Proposition~\ref{prop:int-graphs}(d)}

The final part of this appendix is devoted to showing that if $\mc{G}$ is a union of few cosets in $S_n$, then $\mc{G}$ has at least as many disjoint pairs as $\mc{T}(n,s)$, where $s=\card{\mc{G}}$. To bound the gap $\disj(\mc{G})-\disj(\mc{T}(n,s))$, we shall use part (III) of the following claim concerning structural properties of intersection graphs.

\begin{claim}\label{clm:intgraphsd} Let $G$ be the intersection graph of a union $\mc{G}$ of at most $cn^{1/2}$ cosets in $S_n$. Then the following properties hold.
\begin{itemize}
	\item[(I)] $k_2 \ge \max \{ k_1, 2k_1-6 \}$ unless one of the following cases occurs:
	\begin{compactitem}
		\item[\rm (i)] $\mc{G}$ is canonical;
		\item[\rm (ii)] $G$ is isomorphic to $2K_2, P_4, P_5$ or $C_4\cup K_1$.
	\end{compactitem}
	\item[(II)] If $k_2 \ge k_1$, then
	\[
	k_2 \left(k_2-k_1+1\right)\ge 2k_3+1.
	\] 
	\item[(III)] If $\mc{G}$ is not canonical, then
	\[
	k_2 \left(k_2 -k_1 +1\right)+i(\bar{P_3},G)-k_3 \ge \tfrac{1}{50}k_1 k_2.
	\] 
\end{itemize}
\end{claim}

We note that parts (I) and (II) will only be used to prove part (III).  Before proving this claim, we show how it implies the final part of the proposition.

\begin{proof}[Proof of Proposition~\ref{prop:int-graphs}(d)]
If $\mc{G}$ is canonical, then $\mc{G}$ is a union of $k_1-1$ pairwise disjoint cosets and an intersecting family, and so $\disj(\mc{G})=\disj(\mc{T}(n,s))$.

Now suppose that $\mc{G}$ is not canonical. By Proposition~\ref{prop:int-graphs}(a), and since $k_4 \le k_1^2 k_2$,
	\[
	s=\card{\mc{G}} = k_1 (n-1)!- k_2 (n-2)!+ k_3 (n-3)!\pm k_1^2 k_2 (n-4)!.
	\]
	So if we write $s=:(k_1 +\ep)(n-1)!$, then 
	\[
	\ep(n-1)!=s-k_1(n-1)!=-k_2 (n-2)!+k_3 (n-3)!\pm k_1^2 k_2 (n-4)!,
	\] 
	which is non-negative since $k_3 \le k_1 k_2$ and $k_1 \le cn^{1/2}$.  Using~\eqref{est-T} with $\ep \le 0$ yields
	\[
	\disj(\mc{T}(n,s))=\sum_{1\le i \le 3}b_i (n-i)!D_{n-1} \pm k_1^3 k_2 (n-1)!(n-4)!,
	\]
	where $b_1=\binom{k_1}{2}, b_2=-(k_1-1)k_2$ and $b_3=(k_1-1)k_3$.
	We can easily derive from this and Claim~\ref{clm:intgraphsc} that
	\[ \disj(\mc{G})-\disj(\mc{T}(n,s))=\tfrac12 \left[ k_2(k_2-k_1+1)+i({\bar{P_3}},G)-k_3\right](n-3)!D_{n-1} \pm 13 k_1^3 k_2 (n-1)!(n-4)!. \]
	Furthermore, we have $k_2 (k_2- k_1 +1)+i(\bar{P_3},G)-k_3(G) \ge \tfrac{1}{50} k_1 k_2$ by Claim~\ref{clm:intgraphsd}(III). Therefore,
	\[
	\disj(\mc{G})-\disj(\mc{T}(n,s)) \ge \tfrac{1}{100}k_1 k_2 (n-3)!D_{n-1}-13 k_1^3 k_2 (n-1)!(n-4)!>0,
	\]
	as $D_{n-1}=(e^{-1}+o(1))(n-1)!$, $1\le k_1 \le cn^{1/2}$, and, since $\mc G$ is not canonical, $k_2 \ge 1$.
\end{proof}

Thus to complete the proof of Proposition~\ref{prop:int-graphs}, we need to prove Claim~\ref{clm:intgraphsd}.  The first part shows that, but for a handful of small exceptions, the intersection graph of a non-canonical union of cosets must have many edges.

\begin{proof}[Proof of Claim~\ref{clm:intgraphsd}(I)]
	It is not difficult to verify the result for $k_1 \le 5$. It remains to deal with the case that $k_1\ge 6$ and $k_2<\max\{k_1,2k_1-6\}=2k_1-6$, in which case we wish to show $\mc{G}$ to be canonical.
	
	Let $\l$ be an axis-aligned line that maximises $d:=\card{\l \cap V(G)}$. If $d\ge k_1 -1$, then $\mc{G}$ is canonical, as desired. If $d \le 2$, then $d_G(x) \ge k_1-3$ for every $x\in V(G)$. Hence, as $k_1 \ge 6$,
	\[ k_2 \ge \frac12 k_1 ( k_1 -3) \ge 3( k_1 -3)>2 k_1 -6, \]
	a contradiction. 
	
	We may therefore assume $3 \le d \le k_1-2$. Since each vertex $x\in V(G) \setminus \l$ is incident to all but at most one vertex in $\l$, we must have 
	\[ k_2 \ge (k_1-d)(d-1)\ge 2(k_1-3), \] 
	giving the required contradiction.
\end{proof}

The next part of the claim bounds the number of triangles in terms of the number of edges and vertices.

\begin{proof}[Proof of Claim~\ref{clm:intgraphsd}(II)]	
	We use induction on $k_1$.  The cases $k_1 \le 6$ can be checked by hand.
	
	Now suppose $k_1 \ge 7$. If $k_2 \ge k_1$, $\mc G$ cannot be canonical.  It then follows from part (I) that 
	\begin{equation}\label{eq-eG}
	k_2 \ge 2 k_1 -6\ge k_1 +1.
	\end{equation}
	Let $x$ be a vertex of $G$ of minimum degree. We distinguish two cases.
	
	\noindent \textbf{Case 1:} $x$ is isolated. In this case, vertices of $G$ must lie entirely in the two axis-aligned lines $\l_1$ and $\l_2$ passing through $x$, and thus $G$ is bipartite, implying $k_3=0$. As a consequence, \[ k_2 \left( k_2-k_1+1\right)\stackrel{(\ref{eq-eG})}{\ge}2(k_1+1) \ge 16>2k_3+1.\]
	
	\noindent \textbf{Case 2:} $d_G(x)\ge 1$. Let $G':=G-\{x\}$.
	Then, as $x$ is of minimum degree in $G$, 
	\[k_2(G')\ge k_2(G)-\frac{2k_2(G)}{k_1(G)}\stackrel{(\ref{eq-eG})}{>}k_1(G)-2.\]
	Thus $k_2(G') \ge k_1(G)-1 = k_1(G')$, and so the induction hypothesis applies to $G'$.  Note that
	\begin{align*}
		k_2(G)\left(k_2(G)-k_1(G)+1\right) &= (k_2(G')+d_G(x))(k_2(G')-k_1(G')+1+d_G(x)-1)\\
		&\ge k_2(G')(k_2(G')-k_1(G')+1)+d_G(x)(d_G(x)-1)
	\end{align*}
	since $d_G(x) \ge 1$ and $k_2(G') \ge k_1(G')$.  By the induction hypothesis, $k_2(G')(k_2(G')-k_1(G')+1) \ge 2k_3(G')+1$, and since there are at most $\binom{d_G(x)}{2}$ triangles in $G$ containing $x$, the right hand side of the above expression is at least $2k_3(G) + 1$.
\end{proof}

At long last, this brings us to the final proof of this paper,\footnote{We applaud the reader for making it this far.} the crucial inequality in the proof of Proposition~\ref{prop:int-graphs}(d).

\begin{proof}[Proof of Claim~\ref{clm:intgraphsd}(III)]
	
	If $k_2<k_1$, then, by part (I), $G$ is isomorphic to $2K_2, P_4, P_5$ or $C_4\cup K_1$. We can easily check that $k_2 \left(k_2-k_1+1\right)+i(\bar{P_3},G)-k_3\ge \tfrac{1}{50} k_1 k_2$ in these cases.\footnote{Observe that this is where we require the term $i(\bar{P_3}, G)$; in all other cases we simply use the fact that this is non-negative.}
	
	Suppose, then, that $k_2 \ge k_1$. If $k_1 \le 5$, then by part (II) we have 
	\[
	k_2\left(k_2-k_1+1\right)+i(\bar{P_3},G)-k_3 \ge 1 \ge \tfrac{1}{50}k_1 k_2,
	\]
	as desired.	It remains to handle the case $k_2 \ge k_1 \ge 6$. Part (I) implies $k_2 \ge 2k_1 -6$, and so $k_2-k_1+1\ge \tfrac{1}{6}k_1$. Combining this estimate with part (II), we find
	\[
	k_2\left(k_2-k_1+1\right)+i(\bar{P_3},G)-k_3 \ge \tfrac{1}{2}k_2 (k_2-k_1+1)\ge \tfrac{1}{12}k_1 k_2,
	\] 	
	finishing the proof. 
\end{proof}

\end{document}